\documentclass[preprint]{elsarticle}
\usepackage{amssymb}
\usepackage{tikz}
\usepackage{tikz-cd}
\usepackage{amsmath}
\usepackage{latexsym}
\usepackage{mathrsfs}
\usepackage{amsthm}
\usepackage{verbatim}
\usepackage{graphicx}
\usepackage{epstopdf}
\usepackage{epsfig}
\usepackage{xspace}
\usepackage{placeins}
\usepackage{tabularx}

\usepackage{color}
\usepackage{float}
\usepackage{bm}
\usepackage[colorlinks,linkcolor=blue,anchorcolor=green,citecolor=red]{hyperref}
\usepackage{amsfonts}
\usepackage{fancyhdr}
\usepackage{amscd}
\usepackage{booktabs}

\usepackage[utf8]{inputenc}

\usepackage{url}
\usepackage{hyperref}
\definecolor{red}{rgb}{1,0,0}
\newcommand\red[1]{\textcolor{red}{#1}}
\definecolor{blue}{rgb}{0,0,1}
\newcommand\blue[1]{\textcolor{blue}{#1}}

\newtheorem{theorem}{Theorem}
\newtheorem{remark}{Remark}

\newtheorem{lemma}{Lemma}
\newtheorem{problem}{Problem}

\newtheorem{algorithm}{Algorithm}

\newcommand\grad{\operatorname{grad}}
\renewcommand\div{\operatorname{div}}
\newcommand\curl{\operatorname{curl}}

\newcommand{\scurl}{\curl}
\DeclareMathOperator{\vcurl}{\mathbf{curl}}

\newcommand{\Hhc}{H^{h}_{0}(\curl)}
\newcommand{\Ltz}{L^2_0(\Omega)}
\newcommand{\Hhd}{H^{h}_{0}(\div)}
\newcommand{\Hd}{{H}(\mathrm{div}; \Omega)}
\newcommand{\Hc}{{H}(\mathrm{curl}; \Omega)}

\renewcommand\S{{S}}

\newcommand\R{\mathbb{R}}

\newcommand\A{{\mathcal A}}

\renewcommand{\Re}{\ensuremath{\mathrm{Re}}}
\newcommand{\Rem}{\ensuremath{\mathrm{Re_m}}}
\newcommand{\Reminv}{\ensuremath{\mathrm{Re_m^{-1}}}}
\newcommand{\RH}{\ensuremath{\mathrm{R_H}}}

\newcommand{\Qc}{\mathbb{Q}_{c}}
\newcommand{\Qd}{\mathbb{Q}_{d}}

\newcommand{\MM}{\mathcal{M}}

\newcommand{\NN}{\mathcal{N}}

\newcommand{\CC}{\mathcal{C}}
\newcommand{\DD}{\mathcal{D}}
\renewcommand{\AA}{\mathcal{A}}
\newcommand{\BB}{\mathcal{B}}
\newcommand{\GG}{\mathcal{G}}
\newcommand{\FF}{\mathcal{F}}
\newcommand{\KK}{\mathcal{K}}
\newcommand{\LL}{\mathcal{L}}

\newcommand{\B}{\mathbf{B}}
\newcommand{\E}{\mathbf{E}}
\renewcommand{\u}{\mathbf{u}}
\renewcommand{\v}{\mathbf{v}}
\renewcommand{\j}{\mathbf{j}}
\renewcommand{\k}{\mathbf{k}}
\newcommand{\F}{\mathbf{F}}
\newcommand{\C}{\mathbf{C}}
\newcommand{\n}{\mathbf{n}}

\newcommand{\zerov}{\mathbf{0}}
\newcommand{\f}{\mathbf{f}}
\newcommand{\V}{\mathbf{V}}

\newcommand{\W}{\mathbf{W}}

\makeatletter
\renewcommand*\env@matrix[1][\arraystretch]{%
	\edef\arraystretch{#1}%
	\hskip -\arraycolsep
	\let\@ifnextchar\new@ifnextchar
	\array{*\c@MaxMatrixCols c}}
\makeatother

\DeclareUnicodeCharacter{202F}{\,}

\makeatletter
\@addtoreset{equation}{section}
\makeatother

\usepackage{geometry}
\begin{document}
	
		\begin{frontmatter}
		\title{Structure-preserving and helicity-conserving finite element approximations and preconditioning for the Hall MHD equations}
        \tnotetext[mytitlenote]{This work was funded by the Engineering and Physical Sciences Research Council (grant numbers EP/V001493/1, EP/R029423/1, and EP/W026163/1) and by the EPSRC Centre for Doctoral Training in Partial Differential Equations: Analysis and Applications, grant EP/L015811/1. KH was funded by a Hooke Fellowship.}
		\author{Fabian~Laakmann\corref{mycorrespondingauthor}}
		\cortext[mycorrespondingauthor]{Corresponding author}
		\ead{fabian.laakmann@maths.ox.ac.uk}
		
		\author {Kaibo~Hu}
		\ead{kaibo.hu@maths.ox.ac.uk}
		
		\author {Patrick~E.~Farrell}
		\address{Mathematical Institute, University of Oxford, Oxford, UK}
		\ead{patrick.farrell@maths.ox.ac.uk}

		\begin{abstract}
			We develop structure-preserving finite element methods for the incompressible, resistive Hall magnetohydrodynamics (MHD) equations. These equations incorporate the Hall current term in Ohm's law and provide a more appropriate description of fully ionized plasmas than the standard MHD equations on length scales close to or smaller than the ion skin depth. We introduce a stationary discrete variational formulation of Hall MHD that enforces the magnetic Gauss's law exactly (up to solver tolerances) and prove the well-posedness and convergence of a Picard linearization. For the transient problem, we present time discretizations that preserve the energy and magnetic and hybrid helicity precisely in the ideal limit for two types of boundary conditions. Additionally, we present an augmented Lagrangian preconditioning technique for both the stationary and transient cases. We confirm our findings with several numerical experiments.
		\end{abstract}
	\end{frontmatter}




   
\section{Introduction}   
We consider finite element methods for the solution of the incompressible, resistive Hall magnetohydrodynamics (MHD) equations. The stationary formulation on a bounded polyhedral Lipschitz domain $\Omega \subset \mathbb{R}^3$ is given by
\begin{subequations}
	\label{eq:HallMHD}
\begin{align} \label{eq:HallMHDu}
	- \Re^{-1} \Delta \u + 
 ( \u \cdot \nabla) \u
- \S\, \j \times \B
+ \nabla p
&= \f , \\\label{eq:HallMHDj}
 \j - 
  \nabla \times \B
&= \mathbf{0} , \\\label{eq:HallMHDB}
\nabla \times \E &= \mathbf{0} , \\
 \nabla \cdot \B &= 0, \\
 \nabla \cdot \u &= 0, \\
 \label{eq:HallMHDE}
\Reminv \j - (\E + \u \times \B-\RH\,\j\times \B) & = \mathbf{0} .
\end{align}
\end{subequations}
Here, $\u$ is the fluid velocity, $p$ the fluid pressure, $\j$ the current density, $\E$ the electric field and $\B$ the magnetic field. $\Re$ denotes the fluid Reynolds number, $\Rem$ the magnetic Reynolds number, $S$ the coupling number, $\RH$ the Hall coefficient, and $\f: \Omega \to \R^3$ a source term. When the Hall current term $\RH\,\j\times \B$ vanishes, one obtains the well-known resistive MHD system \cite{Gerbeau2006}.
For time-dependent problems, the time derivatives $\frac{d}{dt}\u$ and $\frac{d}{dt}\B$ are added to the left-hand sides of \eqref{eq:HallMHDu} and \eqref{eq:HallMHDB} respectively. 
We mainly consider the boundary conditions
\begin{equation}
\label{boundary_cond}
\u = \mathbf{0}, \quad \B \cdot \n = 0, \quad \E \times \n = \mathbf{0}, \quad \j \times \n = \mathbf{0}, \quad \text{on $\partial \Omega$},
\end{equation}
where $\n$ is the unit normal vector of $\partial \Omega$.
However, a treatment of the alternative boundary conditions (c.f., \cite{gunzburger1991existence})\begin{align}
\label{boundary_cond2}
\u = \mathbf{0}, \quad \B \times \n = \mathbf{0}, \quad \E \cdot \n = 0, \quad \j \times \n = \mathbf{0}, \quad \text{on $\partial \Omega$}
\end{align}
 is also possible.

The inclusion of the Hall effect provides a more appropriate description of fully ionized plasmas than standard MHD models on length scales close to or smaller than the ion skin depth \cite{Galtier2015}. On these length scales the Hall MHD equations take into account the different motions of ions and electrons in a two-fluid approach. While the electron motion is frozen to the magnetic field in this regime, it remains to solve for the ion fluid velocity $\u$ \cite{Huba2003}. 
The Hall MHD equations can be used to describe many important plasma phenomena, such as magnetic reconnection processes \cite{Forbes1991, Morales2005}, the expansion of sub-Alfv\'{e}nic plasma \cite{Ripin1993} and the dynamics of Hall drift waves and Whistler waves \cite{Huba2003}.  

The essence of the Hall effect is described by adding the Hall-term $\j \times \B$ in the generalized Ohm's law \cite[Section 2.2.2]{Galtier2015}
\begin{equation}
	\eta \j = \E + \u \times \B - \frac{1}{n e} \j \times \B,
\end{equation}
where $\eta$ denotes the magnetic resistivity, $n$ the charge density and $e$ the electron charge. The non-dimensionalized form of the generalized Ohm's law corresponds to \eqref{eq:HallMHDE} where the Hall parameter $\RH$ is defined as 
\begin{equation}
	\RH = \frac{1}{\mu_0 n e} \frac{\overline{B}}{L \overline{U}}
\end{equation}
for a characteristic length $L$, magnetic field strength $\overline{B}$ and speed $\overline{U}$ of the fluid, and the vacuum permeability $\mu_0$. We refer to the case $\RH = 0$ as the standard MHD equations.

Several analytical results for the continuous Hall MHD problem \cite{Chae2014, Danchin2020} and computational results of physical simulations \cite{Gmez2008, Chacn2003, Tth2008} are available in the literature. However, little attention has been paid to provide well-posedness and convergence results for the numerical approximation of these equations. 
We aim to contribute to this field by introducing a variational formulation and structure-preserving discretization for the stationary and time-dependent Hall MHD equations and proving a well-posedness and convergence result for a Picard linearization of this formulation. We next construct numerical schemes that preserve the energy, magnetic helicity and hybrid helicity precisely in the ideal limit of $\Re=\Rem=\infty$.
Finally, we investigate parameter-robust preconditioners for the efficient solution of the arising linear systems.

Although \eqref{eq:HallMHD} only differs by one term from the standard MHD equations, the extension of existing theory and algorithms to the Hall case is non-trivial. Most formulations of the standard MHD equations use Ohm's law \eqref{eq:HallMHDE} to eliminate $\j$ as an unknown; with $\RH \neq 0$ this is no longer possible. Therefore our proposed variational formulation for the stationary problem includes as unknowns both the current density $\j$ and the electric field $\E$.
In the time-dependent case, the various conservation properties of the MHD system in the ideal limit are based upon the symmetries of the system; the introduction of the Hall term changes these symmetries, thus making it substantially more difficult to construct numerical methods that preserve several quantities simultaneously.
 Finally, the development of preconditioning techniques becomes more difficult as an additional non-symmetric term with a non-trivial kernel enters the system.

Most variational formulations of the standard MHD system either look for the magnetic field $\B$ in an $\Hc$- or $\Hd$-conforming space. $\Hc$-conforming formulations have the advantage that they usually include the fewest unknown variables, typically $\u$, $p$, $\B$ and a Lagrange multiplier for the enforcement of the magnetic Gauss's law. However, such formulations only enforce the magnetic Gauss's law weakly, which can cause problems for numerical approximations \cite{Brackbill1980}. Therefore, in recent years much interest has been paid to structure-preserving $\Hd$-conforming approximations that enforce $\nabla \cdot \B = 0$ precisely on the discrete level \cite{Hu20162,hu2019structure}. These formulations either eliminate $\E$ or $\j$ with help of \eqref{eq:HallMHDE} or \eqref{eq:HallMHDj}. Here, the augmented Lagrangian formulation in \cite{hu2020convergence} seems a natural approach, as it only includes $\u,p,\B$ and $\E$ as unknowns and enforces $\nabla \cdot \B = 0$ without the need for a Lagrange multiplier. Our proposed formulation with both $\j$ and $\E$ as unknowns tries to use the fewest number of unknown variables for the Hall system while still enforcing the magnetic Gauss's law precisely.

Another way of enforcing $\nabla \cdot \B = 0$ for the incompressible MHD system is to use formulations based on the vector potential $\mathbf{A}$ where $\B=\nabla \times \mathbf{A}$ (see, e.g., \cite{adler2018vector, hiptmair2018fully, pagliantini2016computational}). The Hall term is $\j\times \B=(\nabla\times \B)\times \B=(\nabla\times \nabla\times \mathbf{A})\times (\nabla\times \mathbf{A})$, which is a high order term in $\mathbf{A}$. It seems difficult to deal with this term with the magnetic potential and we will not pursue potential-based formulations in this work. 


The ideal limit in Hall MHD describes the case of vanishing magnetic resistivity $\eta$. We also include the case of vanishing fluid viscosity $\nu$ in this notion and hence the ideal limit formally corresponds to $\Re=\Rem=\infty$. It is well-known that in this case the energy, magnetic helicity and cross helicity are conserved properties of the standard MHD system \cite{Galtier2015}. For the ideal Hall MHD system, the cross helicity is not conserved any more; instead the so-called hybrid helicity \cite{Mininni_2003} is conserved, which is a suitable combination of magnetic, cross and fluid helicity. In \cite{gawlik2020} and \cite{hu2021helicity}, the authors propose numerical algorithms that preserve the conservative properties of the standard MHD equations precisely on the discrete level. We extend their work for the additional Hall term and propose algorithms that also preserve the hybrid helicity precisely.

Helicity characterises the linkage of field lines (the vortex lines for the fluid helicity, the magnetic lines for the magnetic helicity etc.), and is thus fundamentally important for the flow kinematics \cite{moffatt1992helicity}. 
The importance of the magnetic and cross helicity can be found in, e.g., \cite{taylor1974relaxation,pariat2005photospheric,perez2009role} and the references therein. 
Even in the non-ideal case, i.e., for non-vanishing resistivity, the total helicity is approximately preserved if the magnetic fields undergoes small-scale turbulence \cite[Remark 7.19]{arnold1999topological}. Hence, algorithms that preserve the helicity and other quantities precisely (or nearly in the non-ideal case) at the discrete level are important and can lead to more physical solutions for the same resolution.

The development of preconditioning strategies for the standard MHD equations is a field of active research. Common approaches are either based on block preconditioners \cite{laakmann2021, Phillips2016, Wathen2020} or fully-coupled multigrid methods \cite{adler2020monolithic, shadid2016scalable}. The solver proposed in \cite{laakmann2021} achieves good robustness with respect to the Reynolds and coupling numbers. Their approach is based on a Schur complement approximation of the resulting block system and the use of parameter-robust multigrid methods for the electromagnetic and hydrodynamic blocks. We extend this approach to the Hall MHD system.
The range of the Hall parameters is typically between 0 and 1 and there exists many numerical simulations in the literature that consider the effect of different values of the Hall parameters in this range \cite{Donato2012, Morales2005, Shi2019}.

The remainder of this work is outlined as follows. In Section \ref{sec:vatiational-formulation}, we derive a variational formulation of the stationary Hall MHD system and prove the well-posedness of a Picard linearization. In Section \ref{sec:timedepproblems}, we derive time discretizations for the transient problem that preserve the energy, magnetic and hybrid helicity precisely in the ideal limit. An augmented Lagrangian preconditioner for the Hall MHD system is derived in Section \ref{sec:ALP}. Finally, we present numerical results in Section \ref{sec:numericalresults}, which include iterations numbers for a lid-driven cavity problem, the simulation of magnetic reconnection for an island coalescence problem and a numerical verification of the conservation properties for our algorithms in the ideal limit.

\section{Stationary variational formulation, linearization and discretization}\label{sec:vatiational-formulation}

\subsection{Preliminaries and notation}
We assume that $\Omega$ is a bounded Lipschitz polyhedron. For the ease of exposition, we further assume that $\Omega$ is contractible. 
We use $(\cdot,\cdot)$ and $\|\cdot\|$ to denote the $L^2(\Omega)$ inner product and norm. The dual pairing between an $H^{-1}$ (with norm $\|\cdot\|_{-1}$) and $H^1$ (with norm $\|\cdot\|_{1}$) function is denoted as $\langle \cdot , \cdot \rangle$.
We define the function spaces
$$
H_{0}^1(\Omega):=\left\{v\in H^{1}(\Omega): v=0  \mbox{ on } \partial\Omega =0\right \},
$$
$$
H_0(\curl,\Omega):=\{\bm{v}\in H(\curl, \Omega), \bm{v}\times \bm{n}=0 \mbox{ on } \partial\Omega\},
$$
$$
H_0(\div,\Omega):=\{\bm{w}\in H(\div, \Omega), \bm{w}\cdot \bm{n}=0 \mbox{ on } \partial\Omega\},
$$
and
$$
L^2_0(\Omega):=\left\{v\in L^2(\Omega):  \int_\Omega v=0 \right\}.
$$
We may drop the domain $\Omega$ in the notation of the function spaces if it is obvious which domain we consider.
 We use the finite element de~Rham sequence
\begin{equation}
\label{eqn:derhamfe}
	\begin{tikzcd}
		0\arrow{r} & H^{h}_{0}(\grad) \arrow{r}{\grad} & H^{h}_{0}(\curl) \arrow{r}{\curl} &H^{h}_{0}(\div) \arrow{r}{\div} & L^{2}_{h} \arrow{r}{}&0, 
	\end{tikzcd}
\end{equation}
to discretize the variables, where $H^{h}_{0}(D)\subset H_{0}(D), ~D = \grad, \curl, \div$ are conforming finite element spaces, see e.g.~Arnold, Falk, Winther
\cite{Arnold.D;Falk.R;Winther.R.2006a,Arnold.D;Falk.R;Winther.R.2010a},
Hiptmair \cite{Hiptmair.R.2002a}, Bossavit \cite{Bossavit.A.1998a} for
more detailed discussions on discrete differential forms. 
There are families of finite element de~Rham complexes \eqref{eqn:derhamfe} with any degree. In the schemes presented below, we require that $\E_h, \j_h \in \Hhc $ and $\B_h \in \Hhd $, i.e.~that they are drawn from the same sequence.
We define $\tilde\nabla_{h}\times  : [L^{2}(\Omega)]^{3}\to \Hhc$ by
$$
(\tilde \nabla_{h}\times \B_h, \k_{h})=(\B_h, \nabla\times \k_{h}), \quad\forall\, \k_{h}\in \Hhc.
$$
We denote the finite element spaces used for the velocity $\bm{u}_h$ and pressure $p_h$
by $\bm{V}_{h}$ and $Q_{h}$ respectively,
and assume that the choice is inf-sup stable
\cite{Girault.V;Raviart.P.1986a}.

We regularly use the generalised Gaffney inequality
\begin{equation}
	\|\B_h\|_{L^{3+\delta}} \leq \|\tilde{\nabla} \times \B_h\| + \|\nabla \cdot \B_h\| \quad \forall \ \B_h \in H^{h}_{0}(\div,\Omega)
\end{equation}
for  $0\leq \delta \leq3$, where $\delta$ depends on the regularity of $\Omega$. For a proof, we refer to \cite[Theorem 1]{he2019generalized} and references therein. 

The vorticity is often denoted as $\bm \omega = \nabla \times \u$.

{}

\subsection{Nonlinear scheme}
\label{sec_scheme}

We propose the following variational form for the stationary problem  \eqref{eq:HallMHD} with boundary conditions (\ref{boundary_cond}).
Define $\bm{X}_{h}:=\bm{V}_{h}\times Q_h \times H_{0}^{h}(\curl)\times H_{0}^{h}(\div)\times H_{0}^{h}(\curl)$.
\begin{problem}\label{prob:continuous}
Find $(\u_h, p_h, \E_h, \B_h, \j_h)\in \bm{X}_{h}$, such that for any $(\v_h, q_h, \F_h, \C_h, \k_h)\in \bm{X}_{h} $,
\begin{subequations}\label{fem-discretization}
\begin{align}\label{fem-1}
	\Re^{-1} (\nabla \u_h, \nabla \v_h)  
  +(( \u_h \cdot \nabla) \u_h, \v_h)
  - \S (\j_h\times \B_h,\v_h ) - (p_h,\nabla\cdot \v_h) 
 &= \langle \f,\v_h \rangle,\\
\label{fem-2}
(\j_h,\F_h) -( \B_h, \nabla\times \F_h) &=0,\\
\label{fem-3}
 (\nabla\times \E_h, \C_h) +(\nabla\cdot \B_h, \nabla\cdot \C_h)&= 0,\\ 
\label{fem-4}
 \Reminv(\j_{h}, \k_{h})-(\E_{h}+\u_{h}\times \B_{h}-\RH\,\j_{h}\times \B_{h}, \k_{h})&=0,\\
\label{fem-5}
 -(\nabla\cdot \u_h, q_h) &=0.
 \end{align}
\end{subequations}
\end{problem} 
  
The above formulation includes the weak form of the augmented Lagrangian term $-\nabla \nabla \cdot \B_h$ in \eqref{fem-3}, which is used to enforce the magnetic Gauss's law $\nabla \cdot \B_h=0$ precisely.
We summarize some properties of the variational formulation in the next theorem.
\begin{theorem}\label{stability_fem}
Any solution for  Problem \ref{prob:continuous} satisfies
\begin{enumerate}
\item magnetic Gauss's law:
$$
\nabla\cdot \B_h=0,
$$
\item  stationary Faraday's law:
$$
\nabla\times \E_h =\mathbf{0},
$$
\item energy estimates:
\begin{align}\label{energy-1}
{\Re^{-1}}\|\nabla\u_h\|^{2}+\Reminv \S\|\j_h\|^{2} &= \langle \f, \u_h\rangle, \\
\label{energy-2}
\frac{1}{2}\Re^{-1}\|\nabla\u_h\|^{2}+\Reminv \S\|\j_h\|^{2} &\leq \frac{\Re}{2}\|\f\|_{-1}^{2}.
\end{align}
\end{enumerate}
\end{theorem}

\begin{proof} 

As in \cite{hu2020convergence}, the stationary Faraday's law $\nabla\times \E_{h}=\mathbf{0}$ follows from testing \eqref{fem-3} with $\C_h = \nabla\times \E_h$, and the magnetic Gauss' law $\nabla\cdot \B_h = 0$ then follows from testing \eqref{fem-3} with $\C_h = \B_h$.
The proof of the energy law follows from testing \eqref{fem-4} with $\j_h$. Since the additional Hall term $\RH\,(\j_{h}\times \B_{h}, \k_{h})$ vanishes for $\k_h=\j_h$, the proof coincides with the one in \cite{hu2020convergence} for the standard MHD system.
 \end{proof}

\subsection{Picard iteration}

In the following, we propose a Picard-type iteration for Problem \ref{prob:continuous}. For MHD models, Picard-type iterations have the advantage that they allow rigorous well-posedness proofs, c.f.\ \cite{hu2020convergence}. In this section, we extend these proofs for the additional Hall term. The well-posedness of the full Newton linearization is much more difficult to achieve or even unknown for certain MHD formulations. However, they often show better nonlinear convergence in practice, especially in the regime of high magnetic Reynolds numbers, see \cite{laakmann2021}. In Section \ref{sec:numericalresults}, we report numerical results for both linearization types. 

\begin{algorithm}[Picard step]
\label{alg:picard-s}
Given $(\u_{h}^{n-1},\B_{h}^{n-1})$,   find $(\u_{h}^{n}, p_h^n, \E_{h}^{n}, \B_{h}^{n}, \j_{h}^{n})\in \bm{X}_{h}$, such that for any $(\v_{h}, q_h, \F_{h}, \C_{h}, \k_{h})\in \bm{X}_{h} $,
\begin{subequations}
\begin{align}
\label{picard1}
  {\Re^{-1}} (\nabla \u^{n}_{h}, \nabla \v_{h}) 
  + (( \u^{n-1}_h \cdot \nabla) \u^{n}_h, \v_h)
  - \S (\j_{h}^{n}\times \B^{n-1}_{h},\v_{h} ) - (p_{h}^{n},\nabla\cdot \v_{h}) 
& = \langle \f,\v_{h} \rangle,\\
\label{picard2} 
 (\j_{h}^{n},\F_{h}) - ( \B_{h}^{n}, \nabla\times \F_{h}) &=0, \\
\label{picard3}
(\nabla\times \E_{h}^{n}, \C_{h}) +(\nabla\cdot\B_{h}^n, \nabla\cdot \C_{h})&= 0, \\ 
\label{picard4} 
\Reminv  (\j_{h}^{n}, \k_{h}
 )-(\E^{n}_{h}+\u^{n}_{h}\times \B^{n-1}_{h}-\RH\,\j_{h}^{n}\times \B^{n-1}_{h}, \k_{h})&=0,
\\\label{picard5}
 -(\nabla\cdot \u^{n}_{h}, q_{h}) &=0.
\end{align}
\end{subequations}
\end{algorithm} 

\begin{algorithm}[Newton iteration]\label{alg:newton-s}
	The Newton iteration includes the additional terms $(( \u^{n}_h \cdot \nabla) \u^{n-1}_h, \v_h) + S (\j_h^{n-1} \times \B_h^n, \v_h)$ on the left-hand side of \eqref{picard1}, and $-(\u_h^{n-1}\times \B_h^n, \k_h) + \RH\,(\j_h^{n-1}\times \B_h^n, \k_h)$ on the left hand side of \eqref{picard4}.
\end{algorithm}

\begin{remark}
	By construction, any solution $(\u^n_h, p^n_h, \E^n_h, \B^n_h, \j^n_h)$ of Algorithm \ref{alg:picard-s} also fulfils $(1)$, $(2)$, and $(3)$ from Theorem \ref{stability_fem} precisely.
\end{remark}

We will use the Brezzi theory \cite{Brezzi.F.1974a} to prove the well-posedness of the Picard iteration. 
We recast Algorithm \ref{alg:picard-s} as follows.
We first formally eliminate the variables $\j_{h}^{n}$ and $\E_{h}^{n}$ from the system by 
\begin{equation}\label{jE}
\j_{h}^{n}=\tilde{\nabla}_{h}\times \B_{h}^{n}, \quad \E_{h}^{n}=\Reminv\tilde{\nabla}_{h}\times \B_{h}^{n}-\mathbb{Q}_{c}(\u_{h}^{n}\times \B_{h}^{n-1})+\RH\, \mathbb{Q}_{c}((\tilde{\nabla}_{h}\times \B_{h}^{n})\times  \B_{h}^{n-1}),
\end{equation}
where $\mathbb{Q}_{c}$ is the $L^{2}$ projection to $H^{h}_{0}(\curl, \Omega)$. Then \eqref{picard1}-\eqref{picard5} becomes 
\begin{subequations}
\begin{align}
\label{picard-reduced1}
  {\Re^{-1}} (\nabla \u^{n}_{h}, \nabla \v_{h}) 
  + (( \u^{n-1}_h \cdot \nabla) \u^{n}_h, \v_h)
  - \S ((\tilde{\nabla}_{h}\times \B_{h}^{n})\times \B^{n-1}_{h},\v_{h} ) - (p_{h}^{n},\nabla\cdot \v_{h}) 
 &= \langle \f,\v_{h} \rangle,\\
\label{picard-reduced2}
\Reminv(\tilde{\nabla}_{h}\times \B_{h}^{n}, \tilde{\nabla}_{h}\times \C_{h} )-(\u_{h}^{n}\times \B_{h}^{n-1}, \tilde{\nabla}_{h}\times \C_{h} )\qquad \qquad  \qquad \qquad \qquad \qquad \quad &\nonumber\\ +\RH\,(\tilde{\nabla}_{h}\times \B_{h}^{n})\times  \B_{h}^{n-1}, \tilde{\nabla}_{h}\times \C_{h} ) +(\nabla\cdot\B_{h}^n, \nabla\cdot \C_{h})&= 0, \\ 
\label{picard-reduced3}
 -(\nabla\cdot \u^{n}_{h}, q_{h}) &=0.
\end{align}
\end{subequations}

Define $\bm{W}_{h}:=\bm{V}_{h}\times H_{0}^{h}(\div)$. Given $(\u^{-}, \B^{-})\in \bm{W}_{h}$, for ${\bm{x}}=(\u, \B)$, ${\bm{y}}=(\v, \C)\in \bm{W}_{h}$ and $p, 
q\in Q_{h}$, we define the bilinear forms
\begin{align*}
a({\bm{x}}, {\bm{y}})
 :=&\Re^{-1}(\nabla \u, \nabla \v) + (( \u^{-} \cdot \nabla) \u, \v) - \S ((\tilde{\nabla}_{h}\times \B)\times \B^{-},\v) +(\nabla\cdot\B, \nabla\cdot \C)
\\
&+\Reminv( \tilde{\nabla}_{h}\times \B, \tilde{\nabla}_{h}\times\C)-(\u \times \B^{-}, \tilde{\nabla}_{h}\times \C)+\RH\,((\tilde{\nabla}_{h}\times \B)\times  \B^{-}, \tilde{\nabla}_{h}\times \C),
\end{align*}
$$
b(\bm{x}, q):=(\nabla\cdot \u, q). 
$$

The mixed form of the Picard step in Algorithm \ref{alg:picard-s}
can be written as: for $\bm{h}\in \bm{W}_{h}^{\ast} $ and $g\in Q_{h}^*$, 
find $({\bm{x}}, p)\in  \bm{W}_{h} \times Q_{h}$, such that for all
$({\bm{y}}, q) \in \bm{W}_{h} \times Q_{h}$,  
\begin{alignat}{3}
& a({\bm{x}}, \bm{y})+ b(\bm{y}, p)
&&=\langle \bm{h}, \bm{y} \rangle, \\
& b(\bm{x}, q)
&&=\langle g, q \rangle .
\label{brezzip}
\end{alignat}

Define the norms
\begin{equation}\label{stationary-BE-norm-X}
\|(\u,  \B)\|^{2}_{X}:=\|\nabla \u\|^{2}+\|\nabla\cdot \B\|^{2}+\|\tilde{\nabla}_{h}\times \B\|^{2},
\end{equation}
\begin{equation}\label{norm-Y}
\|p\|_{Q}:=\|p\|.
\end{equation}
We verify that $\|\cdot\|_{X}$ is a norm. Indeed, $\|(\u,   \B)\|_{X}^{2}$ is quadratic for $\bm x:=(\u,  \B)$. Moreover, when $\|(\u,   \B)\|_{X}=0$, we find $\u=\mathbf{0}$ (Poincar\'e inequality) and $\B=\mathbf{0}$ (generalized Poincar\'e inequality or the discrete Gaffney inequality).

\begin{theorem}\label{thm:wellposed-picard}
Assume that $\B^{-}\in L^{\infty}(\Omega)$. Then, problem \eqref{brezzip} is well-posed with the norms defined by \eqref{stationary-BE-norm-X} and \eqref{norm-Y}.
\end{theorem}

\begin{proof}
To prove the well-posedness of \eqref{brezzip} based on the Brezzi theory, we need to verify the boundedness of each term, the inf-sup condition of $b(\cdot, \cdot)$ and the coercivity of $a(\cdot,\cdot)$ on the discrete kernel defined by 
$$
\bm{W}_{h}^{0}:=\{\bm x\in \bm{W}_{h}: ~(\nabla \cdot \u, q)=0 \quad \forall q\in Q_{h}\}.
$$
The boundedness of both bilinear forms is obvious from the definition of the norms. In particular, the Hall term fulfills
$$
|((\tilde{\nabla}_{h}\times \B)\times  \B^{-}, \tilde{\nabla}_{h}\times \C)|\leq \|\tilde{\nabla}_{h}\times \B\|\|\B^{-}\|_{L^{\infty}}\|\tilde{\nabla}_{h}\times \C\|.
$$
The inf-sup condition of $b(\cdot, \cdot)$ follows by assumption.
To prove coercivity on the kernel, we take
$\v=\u$ and $\C=S\B$, yielding
\begin{align*}
\bm a((\u,   \B), (\v,   \C))&= \Re^{-1}\|\nabla \u\|^{2}+ S\|\nabla\cdot \B\|^{2}+ S\Reminv \|\tilde{\nabla}_{h}\times \B\|^{2},
\end{align*}
and thus the coercivity of $a(\cdot, \cdot)$.
Combining the boundedness of the variational forms, the inf-sup condition of $b(\cdot, \cdot)$ and the coercivity of $a(\cdot,\cdot)$ on $\bm{W}_{h}^{0}$, we complete the proof.
\end{proof}
\begin{remark}
The assumption $\B^{-}\in L^{\infty}(\Omega)$ is due to the Hall term, since we do not have higher regularity for $\tilde{\nabla}_{h}\times \B$ and $\tilde{\nabla}_{h}\times \C$ than $L^2(\Omega)$. The other nonlinear terms can be controlled by $\|\cdot\|_{X}$ as, e.g., 
$$
|(\u \times \B^{-}, \tilde{\nabla}_{h}\times \C)|\leq \|\u\|_{L^{6}}\|\B^{-}\|_{L^{3}}\|\tilde{\nabla}_{h}\times \C\|\leq C\|\nabla \u\|(\|\tilde{\nabla}_{h}\times \B^{-}\|^{2}+\|\nabla\cdot \B^{-}\|^{2})^{\frac{1}{2}}\|\tilde{\nabla}_{h}\times \C\|,
$$
where we used the Poincar\'e inequality, the Sobolev embedding, and the discrete Gaffney inequality for the last step. On the discrete level, we always have that the finite element function $\B^{-}\in L^{\infty}(\Omega)$ and hence we have proved the well-posedness of the discrete problem on a fixed mesh.
\end{remark}

\begin{remark}
	In the above proof, we have used that $(( \u^{-} \cdot \nabla) \u, \u)=0$ which holds if $\nabla\cdot\u=0$ is enforced exactly on the discrete level. If one wishes to use a Stokes pair that is not exactly divergence-free, one can replace this term by $(( \u^{-} \cdot \nabla) \u, \v)-(( \u^{-} \cdot \nabla) \v, \u)$. This approximation is equal to  $(( \u^{-} \cdot \nabla) \u, \u)=0$ if $\nabla\cdot\u$ and a consistent approximation otherwise, cmp.~\cite{hu2020convergence}.
\end{remark}

\begin{remark}[Boundary conditions]
	For the standard MHD equations with $\u=\mathbf{0}$ and $\B \cdot \n = \mathbf{0}$ on $\partial \Omega$, the boundary conditions $\E \times \n = \mathbf{0}$ and $\j \times \n = \mathbf{0}$ are equivalent  due to Ohm's law $\j = \E + \u \times \B$. However, for the Hall MHD equations $\E \times \n = \mathbf{0}$ and $\j \times \n = \mathbf{0}$ are independent. The generalized Ohm's law then implies
	\begin{align}
		&\Reminv  \j \times \n = \E\times \n + (\u \times \B) \times \n - \RH\,(\j \times \B) \times \n\\
		\Leftrightarrow\  & \RH\,(\j \times \B) \times \n= \mathbf{0}\\
		\Leftrightarrow\  & \RH \left[(\j \cdot \n) \B - \j\, \B \cdot \n \right] = \mathbf{0}\\
		\Rightarrow\  & \j \cdot \n = 0.
	\end{align}
	Hence, there exists an additional compatibility condition that $\j \cdot \n = 0$. 
	
\end{remark}	

In the following, we consider the convergence of the Picard iteration. 
\begin{theorem} \label{thm:picard_convergence}
	For a fixed mesh drawn from a quasi-uniform sequence (so that the inverse estimates hold) and $f\in [H^{-1}]^{3}$, $\u_{h}^{n}$, $\E_h^n$, $\j_{h}^{n}$, $\B_{h}^{n}$ and $p_{h}^{n}$ from Algorithm \ref{alg:picard-s} converge if $\Rem$ and $\Re$ are small enough.  
\end{theorem}

The proof is similar to \cite[Theorem 7]{hu2019structure}, and we only give a sketch of the proof focusing on the additional Hall term. 
The essence of the proof is to show that one gets a contraction in the errors $\bm e_{u}^{n}:=\u_{h}^{n}-\u_{h}^{n-1}$ and $\bm e_{j}^{n}:=\tilde{\nabla}_{h}\times \B_{h}^{n}-\tilde{\nabla}_{h}\times \B_{h}^{n-1}$, i.e.,
\begin{equation}\label{contraction0}
\frac{1}{2}(	\Re^{-1}\|\nabla \bm e_{u}^{n}\|^{2}+S \Reminv\|\bm e_{j}^{n}\|^{2})\leq \frac{1}{4}(\Re^{-1}\|\nabla \bm e_{u}^{n-1}\|^{2}+S \Reminv\| \bm e_{j}^{n-1}\|^{2})
\end{equation}
 if $\Re$ and $\Rem$ are small enough. One gets an expression for these errors by subtracting the $(n-1)$-th step of 
\eqref{picard-reduced1}-\eqref{picard-reduced3} from the $n$-th step and using the test functions 
$\v_{h}=\bm e_{u}^{n}$  and $\C_{h}=\B_{h}$. 
This gives
\begin{equation}\label{contraction}
\Re^{-1}\|\nabla \bm e_{u}^{n}\|^{2}+S\Reminv \|\bm e_{j}^{n}\|^{2}=(\u_{h}^{n}\times \B_{h}^{n-1}-\u_{h}^{n-1}\times \B_{h}^{n-2}, \bm e_{j}^{n})+\cdots-\RH(\j_{h}^{n}\times \B_{h}^{n-1}-\j_{h}^{n-1}\times \B_{h}^{n-2}, \bm e_{j}^{n}).
\end{equation}
Here we have omitted other terms of the standard MHD system which are treated in detail in  \cite[Theorem 7]{hu2019structure}.
The last term is the Hall term. The first term can be estimated by 
\begin{align*}
|(\u_{h}^{n}\times \B_{h}^{n-1}-\u_{h}^{n-1}\times \B_{h}^{n-2}, \bm e_{j}^{n})|&=|(\bm e_{u}^{n}\times \B_{h}^{n-1},  \bm e_{j}^{n})+(\u_{h}^{n-1}\times  \bm e_{B}^{n-1},  \bm e_{j}^{n})|\\&
\leq C( \| \bm e_{u}^{n}\|_{L^6}\|\B_{h}^{n-1}\|_{L^3}\| \bm e_{j}^{n}\|+ \|\u_{h}^{n-1}\|_{L^6}\| \bm e_{B}^{n-1}\|_{L^3}\|\bm e_{j}^{n}\|)\\&
\leq C(\|\nabla \bm e_{u}^{n}\|^{2}+\|\bm e_{j}^{n}\|^{2}+\|\bm e_{j}^{n-1}\|^{2}),
\end{align*}
where in the last step we have used the Sobolev embedding $\| \bm e_{u}^{n}\|_{L^6}\leq C\|\nabla \bm e_{u}^{n}\|$, the generalised Gaffney inequality $\| \bm e_{B}^{n-1}\|_{L^3}\leq C\|\tilde{\nabla}\times \bm e_{B}^{n-1}\|=C\|\bm e_{j}^{n-1}\|$, and the energy bounds $\|\B_{h}^{n-1}\|_{L^3}\leq C\|\f\|_{-1}$, $\|\u_{h}^{n-1}\|\leq C\|\f\|_{-1}$ ($\|\f\|_{-1}$ is assumed to be a given finite number). For $\Re^{-1}$ and $\Reminv$ large enough, we can move $\|\nabla \bm e_{u}^{n}\|^{2}$ and $\|\nabla \bm e_{j}^{n}\|^{2}$ to the left hand side of \eqref{contraction}. 

The boundedness of the Hall term is more complicated. In fact, for some $0\leq\delta\leq 3$ depending on the domain, 
\begin{align*}
|(\j_{h}^{n}\times \B_{h}^{n-1}-\j_{h}^{n-1}\times \B_{h}^{n-2}, \bm e_{j}^{n})|&=|(\bm e_{j}^{n}\times \B_{h}^{n-1}, \bm e_{j}^{n})+(\j_{h}^{n-1}\times \bm e_{B}^{n-1}, \bm e_{j}^{n})|\\
&=|(\j_{h}^{n-1}\times \bm e_{B}^{n-1}, \bm e_{j}^{n})|\leq C\|\j_{h}^{n-1}\|_{L^{\frac{6+2\delta}{1+\delta}}}\|\bm e_{B}^{n-1}\|_{L^{3+\delta}}\|\bm e_{j}^{n}\|\\
& \leq  Ch^{-\frac{3}{3+\delta}}\|\j_{h}^{n-1}\|\|\bm e_{j}^{n-1}\| \|\bm e_{j}^{n}\|\\&
\leq  Ch^{-\frac{3}{3+\delta}}(\|\bm e_{j}^{n-1}\| ^{2}+\|\bm e_{j}^{n}\|^{2}),
\end{align*}
where we used the inverse estimate, the generalised Gaffney inequality and the energy bound $\|\j_{h}^{n-1}\|\leq C\|\f\|_{-1}$. Again, we move $\|\bm e_{j}^{n}\|^{2}$ to the left hand side of \eqref{contraction} if $\Re^{-1}$ and $\Reminv$ are large enough.
The contraction \eqref{contraction0} proves the convergence of $\u^n_h$ and $\j^n_h$. Note that the convergence of $\j^n_h$ also implies the convergence of $\B^n_h$ since $\| \bm e_{B}^{n}\|\leq C\|\tilde{\nabla}\times \bm e_{B}^{n}\|=C\|\bm e_{j}^{n}\|$.

To show the convergence of $p_{h}^{n}$, we note that from \eqref{picard-reduced1}, 
\begin{align*}
(p_{h}^{n}-p_{h}^{n-1}, \nabla\cdot \v_{h})=\Re^{-1}(\nabla \bm e_{u}^{n}, \nabla \v_{h})+((\bm e_{u}^{n-1}&\cdot \nabla)\u_{h}^{n}, \v_{h})+((\u_{h}^{n-2}\cdot \nabla)\bm e_{u}^{n}, \v_{h})\\&-\S (\bm e_{j}^{n}\times \B_{h}^{n-1}, \v_{h})-\S (\j_{h}^{n-1}\times \bm e_{B}^{n-1}, \v_{h}).
\end{align*}
From the inf-sup condition of the velocity-pressure pair, there exists $\v_{h}$ such that
$$
(p_{h}^{n}-p_{h}^{n-1}, \nabla\cdot\v_{h})\geq C\|p_{h}^{n}-p_{h}^{n-1}\|^{2}, \quad\mbox{and}\quad \|\v_{h}\|_{1}\leq \|p_{h}^{n}-p_{h}^{n-1}\|. 
$$
Taking this $\v_{h}$ as the test function, we get
\begin{align*}
C\|p_{h}^{n}-p_{h}^{n-1}\|^{2}\leq &\Re^{-1}\| \bm e_{u}^{n}\|_{1}\| \v_{h}\|_{1}+\|\bm e_{u}^{n-1}\|_{1}\|\u_{h}^{n}\|_{1} \|\v_{h}\|_{1}+\|\u_{h}^{n-2}\|_{1}\|\bm e_{u}^{n}\|_{1}\|\v_{h}\|_{1}\\&+\S \|\bm e_{j}^{n}\|\|\B_{h}^{n-1}\|_{L^{3}}\| \v_{h}\|_{1}+\S \|\j_{h}^{n-1}\|\| \bm e_{B}^{n-1}\|_{L^{3}}\| \v_{h}\|_{1}.
\end{align*}
Since $\u_{h}^{n}$ converges in $H^{1}(\Omega)$ and $\B_{h}^{n}$ converges in $L^{3}(\Omega)$ (alternatively, $\j_{h}^{n}=\tilde{\nabla}_{h}\times \B_{h}^{n}$ converges in $L^{2}(\Omega)$), we obtain the $L^{2}$-convergence of $p_{h}^{n}$ by the Cauchy-Schwarz inequality.

For the standard MHD equations, the convergence of the electric field $$\E_{h}^{n}=\Reminv\tilde{\nabla}_{h}\times \B_{h}^{n}-\mathbb{Q}_{c}(\u_{h}^{n}\times \B_{h}^{n-1})+\RH\,\mathbb{Q}_{c}((\tilde{\nabla}_{h}\times \B_{h}^{n})\times  \B_{h}^{n-1})$$ follows from the strong convergence of $\B_{h}^{n}$ in $H_0^h(\div)\cap H_{0}^{h}(\curl)\hookrightarrow L^{3+\delta}$  and $\u_{h}^{n}$ in $H^{1}(\Omega)\hookrightarrow L^{6}(\Omega)$. For the convergence of the Hall-term, we can apply the inverse estimate as before.

\begin{remark}
	For the standard MHD system, the condition on the size of $\Re^{-1}$ and $\Reminv$ only depends on $\|\f\|_{-1}$. Due to the Hall term, this condition also involves a factor $h^{-\frac{3}{3+\delta}}$ which might suggest that the convergence of the Picard iteration deteriorates on finer meshes. Theorem \ref{thm:picard_convergence} proves the convergence of the Picard iteration on a fixed mesh.
\end{remark}

\subsection{2.5D Hall MHD formulation}\label{sec:2.5Dform}
In this subsection, we introduce the 2.5-dimensional formulation of \eqref{eq:HallMHD}, which refers to the assumption that vector fields still have three components but derivatives in the $z$-direction vanish. That means we assume that a three-dimensional vector-field can be decomposed into a two-dimensional vector field and scalar field with the notation
\begin{equation}
	\B(x,y,z) = \begin{pmatrix}
		\tilde{\B} (x,y) \\
		B_3(x,y)
	\end{pmatrix}.
\end{equation}
Recall that there exist two different curl operators in two dimensions, \mbox{given by}
\begin{equation}
	\scurl \tilde{\B} = \partial_x B_2 - \partial_y B_1, \qquad  \vcurl B_3 = \begin{pmatrix}
		\partial_y B_3\\
		-\partial_x B_3
	\end{pmatrix}
\end{equation}
that correspond to the cross-products
\begin{equation}
	\tilde{\u} \times \tilde{\B} =u_1 B_2 - u_2 B_1, \qquad
	\tilde{\B}  \times E_3 =\begin{pmatrix}
		B_2 E_3 \\
		-B_1 E_3
	\end{pmatrix}.
\end{equation}
Hence, we can rewrite the three-dimensional cross-product and curl operator as
\begin{equation}
	\j \times \B = \begin{pmatrix}
		\tilde{\j} \times B_3 - \tilde{\B} \times j_3 \\
		\tilde{\j} \times \tilde{\B}
	\end{pmatrix} 
	\quad \text{ and } \quad
	\nabla \times \B = 
	\begin{pmatrix}
		\vcurl B_3 \\
		\scurl \tilde{\B}
	\end{pmatrix}.
\end{equation}

With this notation we are able to rewrite \eqref{eq:HallMHD} on a bounded polygonal Lipschitz domain $\Omega \subset \mathbb{R}^2$ as 

\begin{subequations}
	\label{eq:2.5DHallMHD}
	\begin{align} \label{eq:2.5DHallMHDu}
		- \Re^{-1} \Delta \tilde{\u} + 
		( \tilde{\u} \cdot \tilde{\nabla}) \tilde{\u}
		- \S\, ( \tilde{\j} \times B_3 -  \tilde{\B} \times j_3)
		+ \tilde{\nabla} p 
		&= \tilde{\f} , \\
		- \Re^{-1} \Delta u_3 + 
		( \tilde{\u} \cdot \tilde{\nabla}) u_3
		- \S\, \tilde{\j} \times \tilde{\B}
		&= f_3 , \\
		\tilde{\j} - 
		\vcurl B_3
		&= \mathbf{0} , \\
		j_3 - 
		\scurl \tilde{\B}
		&= 0 , \\
		\vcurl E_3 &= \mathbf{0} , \\
		\scurl \tilde{\E} &= 0 , \\
		\tilde{\nabla} \cdot \tilde{\B} &= 0, \\
		\tilde{\nabla} \cdot \tilde{\u} &= 0, \\
		\Reminv \tilde{\j} - (\tilde{\E} + \tilde{\u} \times B_3 - \tilde{\B} \times u_3 -\RH\,(\tilde{\j}\times B_3 - \tilde{\B} \times j_3)) & = \mathbf{0},  \label{eq:islandcoal-j}\\
		\Reminv j_3 - (E_3 + \tilde{\u} \times \tilde{\B}-\RH\,\tilde{\j}\times \tilde{\B}) & = 0  \label{eq:islandcoal-j3},
	\end{align}
\end{subequations}

subject to the boundary conditions
\begin{equation}
	\tilde{\u} = \mathbf{0}, \,\, u_3 = 0, \,\, \tilde{\B} \cdot \tilde{\n} = 0, \,\, B_3 = 0, \,\, \tilde{\j} \times \tilde{\n} =  \mathbf{0}, \,\, j_3 = 0, \,\, \tilde{\E} \times \tilde{\n} =  \mathbf{0}, \,\, E_3 = 0. 
\end{equation}
For a finite element discretization, as before we can look for $\tilde{\B}_h$ in an $H_{0}^{h}(\div)$-conforming space and for $\tilde{\j}_h$ and $\tilde{\E}_h$ in an $H_{0}^{h}(\curl)$-confirming space. The other components $u_3$, $B_3$, $j_3$ and $E_3$ are approximated in an $H_{0}^{h}(\grad)$-confirming space.

\section{Conservative discretizations for time-dependent problems}\label{sec:timedepproblems}

For time-dependent problems, we include the time derivatives in the formulation for the stationary problem, i.e., we add $\frac{d}{dt}\u_{h}$ to \eqref{fem-1} and $\frac{d}{dt}\B_{h}$ to \eqref{fem-3}. This means we can remove the $(\nabla\cdot \B_{h}, \nabla\cdot \C_{h})$ term, since the magnetic Gauss's law will be automatically preserved in the evolution provided the initial condition is divergence-free \cite{ma2016robust}.

\subsection{Conserved quantities}
In the ideal limit of $\Re=\Rem=\infty$ it is well-known that the energy, magnetic helicity and cross helicity are conserved properties of the standard incompressible MHD system \cite{Galtier2015}. The energy is defined as
	\begin{equation}
	        E := \int_\Omega |\u|^2 + S|\B|^2 \ \mathrm{d} x,
	\end{equation}
the magnetic helicity is defined as
	\begin{equation}
		H_M := \int_\Omega \mathbf{A} \cdot \B \ \mathrm{d} x
	\end{equation}
for a vector potential $\mathbf{A}$ such that $\nabla \times \mathbf{A} = \B$, and the cross helicity is defined as
	\begin{equation}
	    H_C := \int_\Omega \mathbf{\u} \cdot \B \ \mathrm{d} x.
	\end{equation}
For the ideal Hall MHD equations, the energy and magnetic helicity are still conserved, while the cross helicity is not. Here, hybrid helicity replaces the cross helicity as a conserved property and is defined as 
\begin{equation}
H_H := \int_\Omega (\mathbf{A} + \alpha \u)\cdot (\B + \beta \nabla \times \u) \, \mathrm{d}x
\end{equation}
 for $\alpha$ and $\beta$ satisfying the relation 
  \begin{equation}\label{eqn:alphabeta}
 	2S\alpha\beta-{\RH}(\alpha+\beta)=0.
 \end{equation}
We prove the conservation of hybrid helicity in the next theorem. Note, that the hybrid helicity is a combination of the magnetic, cross and fluid helicity, which is defined as 
\begin{equation}
	H_F := \int_\Omega \u \cdot \nabla \times \u \ \mathrm{d}x.
\end{equation} If ${\RH}=0$, i.e., when the Hall term vanishes, the above equality \eqref{eqn:alphabeta} holds if $\alpha=0$ or $\beta=0$. For $\alpha=\beta=0$,  the hybrid helicity is just the magnetic helicity. If $\alpha=0$ and  $\beta\neq 0$ (alternatively, $\alpha\neq 0$ and $\beta=0$), the hybrid helicity becomes a combination of magnetic and cross helicity. Thus the conservation of hybrid helicity implies the conservation of both magnetic and cross helicity in standard MHD. In Hall MHD, $\alpha=\beta=0$ still corresponds to the magnetic helicity. But in this case \eqref{eqn:alphabeta} does not allow the case $\alpha=0$, $\beta\neq 0$, or $\alpha\neq 0$, $\beta=0$. This means that the cross helicity is not conserved. There exist many non-trivial choices of $\alpha$ and $\beta$, for example, $\alpha=\beta = S^{-1}\RH$.

\begin{theorem}
The generalized hybrid helicity $H_{H}$ is conserved in the time-dependent Hall MHD system with $\f=\mathbf{0}$ and formally $\Re^{-1}=\Reminv =0$ for any $\alpha$, $\beta$ such that \eqref{eqn:alphabeta} holds.
\end{theorem}

\begin{proof}
We have
\begin{align*}
\frac{d}{dt}H_{H}&=\frac{d}{dt}(\mathbf{A}, \B)+\frac{d}{dt}[\alpha(\u, \B)+\beta(\mathbf{A}, \bm \omega)]+\frac{d}{dt}\alpha\beta(\u, \bm \omega)\\
&=\frac{d}{dt}(\mathbf{A}, \B)+\frac{d}{dt}(\alpha+\beta)(\u, \B)+\frac{d}{dt}\alpha\beta(\u,  \bm \omega).
\end{align*}
First, the magnetic helicity is conserved, i.e., 
\begin{align}
	\frac{d}{dt}(\mathbf{A}, \B) &= 2 (\B_t, \mathbf{A}) = 2 (\nabla \times [\u \times \B] , \mathbf{A}) - 2 \RH (\j \times \B , \mathbf{A})\\
	& = 2 (\u \times \B, \B) - 2\RH (\j \times \B, \B) = 0.
\end{align}
 It remains to check the other two terms. In fact,
from \eqref{eq:HallMHDu},
\begin{align*}
(\u_{t}, \B)=(\B, \u\times \bm \omega+S\j\times \B-\nabla p)=(\B, \u\times \bm \omega).
\end{align*}
From \eqref{eq:HallMHDB}, 
\begin{align*}
(\B_{t}, \u)=-(\nabla\times \E, \u)=-(\E, \nabla\times\u)=(\u\times\B-{\RH}\j\times \B, \nabla\times\u)=(\u\times\B,\bm \omega)-{\RH}(\j\times \B, \bm \omega).
\end{align*}
Consequently,
\begin{align*}
\frac{d}{dt}(\alpha+\beta)(\u, \B)=(\alpha+\beta)[(\u_{t}, \B)+(\u, \B_{t})]=-{\RH}(\alpha+\beta)(\j\times \B, \bm \omega).
\end{align*}
Moreover, 
\begin{align*}
\frac{d}{dt}\alpha\beta(\u, \bm \omega)=2\alpha\beta(\u_{t}, \bm \omega)=2\alpha\beta(\u\times \bm \omega+S\j\times \B-\nabla p, \bm \omega)=2S\alpha\beta(\j\times \B, \bm \omega).
\end{align*}
This implies that
$$
\frac{d}{dt}H_{H}=[2S\alpha\beta-{\RH}(\alpha+\beta)](\j\times \B, \bm \omega)
$$
and proves the desired result. 
\end{proof}

Similar to the discussions in \cite{arnold1999topological}, we show that the hybrid helicity provides a lower bound for the energy when $\alpha=\beta=S^{-1}\RH$. This bound, which was referred to as the Arnold inequality in the case of the magnetic helicity \cite[Section 8]{moffatt2021some}, shows that non-zero hybrid helicity, as a measure of the knottedness, provides a topological barrier which prevents a hybrid energy defined by $\|\B+ S^{-1}\RH \bm \omega\|^{2}$ from decaying below a certain value. The conclusion also holds for dissipative flows where the helicity is not conserved.
\begin{theorem}\label{thm:arnold}
$$
\|\B+ S^{-1}\RH \bm \omega\|^{2}\geq C^{-1}|H_{H}|,
$$
where $C$ is the positive constant in the Poincar\'e inequality.
\end{theorem}
\begin{proof}
\begin{align*}
|H_{H}|= \left | \int (\mathbf{A} + S^{-1}\RH \u)\cdot (\B + S^{-1}\RH \bm \omega)\, dx \right |&\leq \|\mathbf{A} + S^{-1}\RH \u\|\|\B + S^{-1}\RH \bm \omega\|\\
&\leq C\|\B + S^{-1}\RH \bm \omega\|^{2}.
\end{align*}
\end{proof}

Next, we present time discretizations that preserve the above quantities precisely on the discrete level.
The MHD system has delicate differential structures reflected in its various  conserved quantities, e.g., the energy, the magnetic Gauss law, and the magnetic and cross/hybrid helicity. In fact, in the proof of the energy conservation, the Lorentz force and the magnetic convection cancel each other, and the fluid convection cancels itself. For the cross helicity, the fluid and magnetic convection cancel each other, and the Lorentz force cancels itself.
To construct conservative numerical methods, it is important to respect these symmetries on the discrete level. This in turn requires certain algebraic structures among the discrete spaces; for example, to preserve the magnetic Gauss law, we discretize unknowns on discrete de~Rham sequences, as in \eqref{eqn:derhamfe}. The magnetic helicity involves the magnetic field and its potential. Therefore it is largely independent of the fluid discretization. However, the energy law and the conservation of cross/hybrid helicity essentially derive from the symmetric coupling between fluids and electromagnetic fields.  Thus it is not surprising that to preserve them on the discrete level, the finite element spaces for the velocity and pressure (Stokes pairs) have to interplay with the spaces for the electromagnetic fields (de~Rham sequences). 

Therefore, the imposition of the boundary condition $\u=\mathbf{0}$ on $\partial \Omega$ can cause difficulties in designing conservative methods, because the description of all components of $\u$ on the boundary does not fit to the electromagnetic boundary conditions. Hence, the literature distinguishes for the standard MHD system between the boundary conditions $\u \times \n$ \cite{hu2021helicity} and $\u \cdot \n$ \cite{gawlik2020}, where the velocity field $\u$ is discretized with $\Hhc$- and $\Hhd$-conforming finite element spaces respectively. Both schemes conserve the energy, magnetic and cross helicity precisely on the discrete level. In the following, we also focus on these two cases and extend the proposed algorithms for the additional Hall-term and the hybrid helicity.

\subsection{Helicity and energy preserving scheme for $\u \times \n =  \mathbf{0}$}\label{sec:schemeutimesn}
In this section, we present a time discretization that preserves the energy and magnetic and hybrid helicity precisely for the boundary condition $\u \times \n = \mathbf{0}$ on $\partial \Omega$. Since these quantities are only preserved for $\f = \mathbf{0}$ and formally $\Re^{-1} = \Reminv =\infty$, we focus only on this case from now on for this section.

The following approach is mainly taken from \cite{hu2021helicity}, but adapted for the additional Hall-term. Let $\Qc$ denote the projection to $\Hhc$, $\Qd$ the projection to $\Hhd$ and $P_h:= p_h + 1/2 |\u_h|^2$ the total pressure. 

We first consider a semi-discrete formulation, discretized in space. We formally eliminate the electric field $\E_h$ by the generalized Ohm's law \eqref{eq:HallMHDE}.
The problem is: find $(\u_h(t), P_h(t), \B_h(t), \j_h(t)) \in \Hhc \times  H^1_0(\Omega) \times \Hhd \times \Hhc$ such that (we drop the argument $t$ in the following)
\begin{subequations}\label{alg:helicity-operator2}
	\begin{align}
		((\u_h)_t, \v_h) + (\Qc[\nabla \times \u_h] \times \u_h, \v_h) - S (\j_h \times \Qc \B_h ,\v_h) + (\nabla P_h, \v_h) = 0 &\quad \forall\, \v_h \in \Hhc,  \label{alg:helicity-operator2-u}\\
		( \u_h, \nabla Q_h) = 0 &\quad \forall\, Q_h \in H^1_0(\Omega), \\
		((\B_h)_t, \C_h) - (\nabla \times \Qc[\u_h \times \Qc \B_h], \C_h ) + \RH (\nabla \times \Qc[\j_h \times \Qc \B_h], \C_h ) = 0    &\quad \forall\, \C_h \in \Hhd, \label{alg:helicity-operator2-b}\\
		(\j_h, \k_h) - (\B_h, \nabla \times \k_h) = 0&\quad \forall\, \k_h \in \Hhc. \label{alg:helicity-operator3-j}
	\end{align}
\end{subequations}
This formulation is useful for analysis but not yet amenable to computation, due to the presence of the projection operators.

\begin{theorem}
	Any solution  $(\u_{h}, p_h, \B_{h}, \j_{h})$ of \eqref{alg:helicity-operator2} fulfils the magnetic Gauss's law $\nabla \cdot \B_h = 0$ precisely if $\nabla \cdot \B_h^0=\mathbf{0}$.
\end{theorem}
\begin{proof}
	 Choosing
	 \begin{equation}
	 	\C_h = (\B_h)_t - \nabla \times \Qc[\u_h \times \Qc \B_h + \RH \j_h \times \Qc \B_h]
	 \end{equation}
     in \eqref{alg:helicity-operator2-b} gives $(\B_h)_t = \nabla \times \Qc[\u_h \times \Qc \B_h + \RH \j_h \times \Qc \B_h] $ and hence $\nabla \cdot \B_h = 0$ if $\nabla \cdot \B_h^0=\mathbf{0}$.
\end{proof}
\begin{theorem}
	Any solution $(\u_{h}, p_h, \B_{h}, \j_{h})$ of \eqref{alg:helicity-operator2} satisfies the energy identity
	$$
	\frac{1}{2}\frac{d}{dt}(\|\u_{h}\|^{2}+S\|\B_{h}\|^{2})= 0.
	$$
\end{theorem}
\begin{proof}
	Testing \eqref{alg:helicity-operator2-u} with $\u_{h}$,
	$$
	\frac{1}{2}\frac{d}{dt}\|\u_{h}\|^{2}= S (\j_h \times \Qc \B_h, \u_h).
	$$
	Testing \eqref{alg:helicity-operator2-b} with $\B_{h}$, 
	\begin{align*}
		\frac{1}{2}\frac{d}{dt}\|\B_{h}\|^{2}
		& = (\nabla \times \Qc[\u_h \times \Qc \B_h], \B_h ) - \RH (\nabla \times \Qc[\j_h \times \Qc \B_h], \B_h ) \\
		& = ( \Qc[\u_h \times \Qc \B_h], \j_h ) - \RH (\Qc[\j_h \times \Qc \B_h], \j_h ) \\
		& = - ( \j_h \times \Qc \B_h, \u_h ).
	\end{align*}
	Here we have used the definition of $\j_h$ in \eqref{alg:helicity-operator3-j} and that $\j_h \in \Hhc$.
	Consequently, the desired result holds by adding the above equalities. 	
\end{proof}

On the discrete level we define the hybrid helicity as 
\begin{equation}
	H_{H}:=\int_\Omega (\mathbf{A}_h + \alpha \u_h)\cdot (\B_h + \beta \bm \omega_h) \, \mathrm{d}x,
\end{equation}
where $\bm \omega_h := \Qc \nabla \times \u_h$ and ($\alpha, \beta$) satisfies \eqref{eqn:alphabeta}. 

\begin{theorem}
	The hybrid helicity of \eqref{alg:helicity-operator2} is conserved if $\f=\mathbf{0}$ and formally $\Re^{-1}=\Reminv =0$ for any $\alpha, \beta$ such that \eqref{eqn:alphabeta} holds.
\end{theorem}

\begin{proof}
	Similar to the continuous level, we have
	\begin{align*}
		\frac{d}{dt}H_{H}=\frac{d}{dt}(\mathbf{A}_{h}, \B_{h})+\frac{d}{dt}(\alpha+\beta)(\u_{h}, \B_{h})+\frac{d}{dt}\alpha\beta(\u_{h}, \bm \omega_h).
	\end{align*}
	
	Testing \eqref{alg:helicity-operator2-b} with $\Qd \mathbf{A}_{h}$, using that $\nabla \times \Hhc \subseteq \Hhd$ and integrating by parts, we have
	\begin{align*}
		\frac{d}{dt}(\mathbf{A}_{h}, \B_{h})=2((\B_{h})_{t}, \mathbf{A}_{h})
		& = 2 (\nabla \times \Qc[ \u_h \times \Qc \B_h], \mathbf{A}_h ) - 2 \RH (\nabla \times \Qc[\j_h \times \Qc \B_h],  \mathbf{A}_h )\\
		& = 2 ( \Qc[\u_h \times \Qc \B_h], \B_h) - 2 \RH (\Qc[\j_h \times \Qc \B_h],  \B_h )\\
		& = 2 ( \u_h \times \Qc \B_h, \Qc \B_h) - 2 \RH (\j_h \times \Qc \B_h,  \Qc\B_h )\\
		&=0.
	\end{align*}
	Testing \eqref{alg:helicity-operator2-u} with $\Qc \B_{h}$, we have
	\begin{align*}
		((\u_{h})_{t}, \B_{h})&=  -(\Qc[\nabla \times \u_h] \times \u_h, \Qc \B_h) + S (\j_h \times \Qc \B_h ,\Qc \B_h) \\
		& = -(\Qc[\nabla \times \u_h] \times \u_h, \Qc \B_h).
	\end{align*}
	Testing \eqref{alg:helicity-operator2-b} with $\Qd \u_{h}$, using that $((\B_{h})_{t}, \Qd \u_{h}) = ((\B_{h})_{t}, \u_{h})$, we have
	\begin{align*}
		((\B_{h})_{t}, \u_{h})
		& = (\nabla \times \Qc[\u_h \times \Qc \B_h], \u_h ) -  \RH (\nabla \times \Qc[\j_h \times \Qc \B_h],  \u_h )\\
		& = (\u_h \times \Qc \B_h, \Qc\nabla \times \u_h ) -  \RH (\j_h \times \Qc \B_h, \Qc[ \nabla \times \u_h]) \\
		& = (\Qc[\nabla \times \u_h] \times \u_h, \Qc \B_h) -  \RH (\j_h \times \Qc \B_h,   \Qc[ \nabla \times \u_h]). 
	\end{align*}
	Consequently,
	\begin{align*}
		\frac{d}{dt}(\alpha+\beta)(\u_{h}, \B_{h})=(\alpha+\beta)[((\u_{h})_{t}, \B_{h})+(\u_{h}, (\B_{h})_{t})]=-{\RH}(\alpha+\beta)(\j_h \times \Qc \B_h, \Qc[ \nabla \times \u_h]) .
	\end{align*}
	Moreover, testing  \eqref{alg:helicity-operator2-u} with $\Qc [\nabla \times \u_{h}]$, we get
	\begin{align*}
		\frac{d}{dt}\alpha\beta(\u_{h}, \nabla \times \u_h)
		&=2\alpha\beta((\u_{h})_{t},  \nabla \times \u_h)\\
		& = -2\alpha\beta(\Qc[\nabla \times \u_h] \times \u_h, \Qc [\nabla \times \u_h]) + 2\alpha\beta S (\j_h \times \Qc \B_h ,\Qc [\nabla \times \u_h])\\
		& = 2\alpha\beta S (\j_h \times \Qc \B_h ,\Qc [\nabla \times \u_h]).\\
	\end{align*}	
	This implies that
	$$
	\frac{d}{dt}H_{H}=[2S\alpha\beta-{\RH}(\alpha+\beta)] (\j_h \times \Qc \B_h, \Qc[ \nabla \times \u_h]).
	$$
\end{proof}

Similar to Theorem \ref{thm:arnold} on the continuous level, we have the following.
The proof is analogous, only using the discrete Poincar\'e inequality \cite[Theorem 5.11]{Arnold.D;Falk.R;Winther.R.2006a}.
\begin{theorem}[discrete Arnold inequality]
$$
\|\B_{h}+ S^{-1}\RH \bm \omega_{h}\|^{2}\geq C^{-1}|H_{H}|,
$$
where $C$ is a positive constant.
\end{theorem}

To render the semi-discrete problem \eqref{alg:helicity-operator2} amenable to computation, we introduce auxiliary variables for the projection operators. The resulting problem is: find $(\u_{h}(t), P_h(t), \B_{h}(t), \E_{h}(t), \j_{h}(t), \mathbf{H}_{h}(t),$ $\bm \omega_{h}(t)) \in$ $\Hhc \times H^1_0(\Omega) \times \Hhd \times [\Hhc]^4 ) $, such that for any $(\v_h, q_h, \C_h, \F_h, \k_h, \bm G_h, \bm \mu_h)$ in the same space,
\begin{subequations}\label{alg:helicityutimesn}
	\begin{align}\label{alg:helicity-cross-u}
		((\u_{h})_{t}, \v_h)- (\u_{h}\times\bm \omega_{h}, \v_h)-S(\j_{h}\times  \bm H_{h}, \v_h)+\Re^{-1}(\nabla\times \u_{h}, \nabla \times \v_h)+( \v_h, \nabla P_{h})&= 0,
		\\
		( \u_{h}, \nabla Q_{h})&=0,
		\\\label{alg:helicity-cross-b}
		((\B_{h})_{t}, \C_h)+(\nabla\times  \E_{h}, \C_h)&=0,\\\label{alg:helicity-cross-j2}
		(\j_{h}, \F_h)-(\B_{h},  \nabla\times \F_h)&=0,\\\label{alg:helicity-cross-h1}
		(\bm H_{h}, \bm G_h)-(\B_{h},  \bm G_h)&=0,\\\label{alg:helicity-cross-h2}
		(\bm \omega_{h}, \bm \mu_h)-(\nabla\times \u_{h},  \bm \mu_h)&=0,\\\label{alg:helicity-cross-w2}
		-\Reminv (\j_{h}, \k_h) + (\E_{h}, \k_h)-((\RH\,\j_{h}-\u_{h})\times \bm H_{h},  \k_h)&=0.
	\end{align}
\end{subequations}
Now \eqref{alg:helicity-cross-j2} gives $\j_{h}=\tilde\nabla_{h}\times \B_{h}$,  \eqref{alg:helicity-cross-h1} gives $\bm H_{h}=\Qc \B_{h}$; and \eqref{alg:helicity-cross-h2} gives $\bm \omega_{h}=\Qc\nabla\times \u_{h}$.

For the time-discretization, we replace the time-derivatives of $(\u_h)_t$ and $(\B_h)_t$ by the difference quotients
\begin{equation}
	D_t \u_h = \frac{\u^{k+1}_h - \u^{k}_h}{\Delta t} \quad \text{ and } \quad D_t \B_h = \frac{\B^{k+1}_h - \B^{k}_h}{\Delta t}.
\end{equation}
We replace $\u_h$ and $\B_h$ with the average of two neighbouring time steps defined as 
$\u^{k+\frac{1}{2}}:=\frac{1}{2}(\u^{k+1}+\u^{k})$ and $\B^{k+\frac{1}{2}}:=\frac{1}{2}(\B^{k+1}+\B^{k})$. All the other auxiliary variables are only defined on the midpoints of two time steps $k+\frac{1}{2}$ (not an average) and denoted as $P^{k+\frac{1}{2}}_h, \E^{k+\frac{1}{2}}_{h}, \j^{k+\frac{1}{2}}_{h}, \mathbf{H}^{k+\frac{1}{2}}_{h}$ and $\bm \omega^{k+\frac{1}{2}}_{h}$. This way we only have to provide initial data $\u^0_h$ and $\B^0_h$ and then solve the time-discretized version of \eqref{alg:helicityutimesn} for each $k \geq 1$; compare with \cite[Algorithm 1]{hu2021helicity}.

\begin{theorem}\label{thm:timeconsutimesn}
	The time-discretized version of \eqref{alg:helicityutimesn} preserves the energy, magnetic and hybrid helicity precisely and enforces $\nabla \cdot \B_h=0$ for all time steps; i.e., for all $ k \geq 0$ there holds
	
\begin{align}
	\int_\Omega \mathbf{u}^{k+1}_h  \cdot \u_h^{k+1} + S \mathbf{B}^{k+1}_h  \cdot \B_h^{k+1} \mathrm{d} x &= \int_\Omega  \mathbf{u}^{k}_h  \cdot \u_h^{k} + S \mathbf{B}^{k}_h  \cdot \B_h^{k} \mathrm{d} x,\\	
	\int_\Omega \mathbf{A}^{k+1}_h  \cdot \B_h^{k+1} \mathrm{d} x &= \int_\Omega \mathbf{A}^{k}_h  \cdot \B_h^{k} \mathrm{d} x,\\
	\int_\Omega \left(\mathbf{A}^{k+1}_h + \alpha \u^{k+1}_h\right)\cdot \left(\B_h^{k+1} + \beta \bm \omega^{k+1/2}_h\right) \mathrm{d} x &= \int_\Omega \left(\mathbf{A}^{k}_h + \alpha \u^{k}_h\right)\cdot \left(\B_h^{k} + \beta \bm \omega^{k-1/2}_h\right) \mathrm{d} x,\\
	\div \B^k_h = 0.
\end{align}
\end{theorem}
\begin{proof}
	These results follow immediately from the proofs of the continuous results by replacing the continuous time-derivative $\partial_t$ by $D_t$. As an example, we prove the conservation of the magnetic helicity. It holds that
	\begin{equation*}
	 \frac{1}{\Delta t}	\int_\Omega \mathbf{A}^{k+1}_h  \cdot \B_h^{k+1} - \mathbf{A}^{k}_h  \cdot \B_h^{k} \ \mathrm{d} x = (D_t \B_h, \mathbf{A}^{k+1/2}_h) + (D_t \mathbf{A}_h, \mathbf{B}^{k+1/2}_h). 
	\end{equation*}
	From the definition of the scheme, it follows that 
	\begin{align*}
		(D_t \B_h, \mathbf{A}^{k+1/2}_h) &= -\left(\nabla \times \E_h^{k+1/2}, \frac{\mathbf{A}_h^{k+1}+\mathbf{A}_h^k}{2}\right) = -\left(\E_h^{k+1/2}, \frac{\mathbf{B}_h^{k+1}+\mathbf{B}_h^k}{2}\right) \\
		& = -\left(\E_h^{k+1/2}, \mathbf{H}^{k+1/2}_h\right) = - \left([\RH \j_h^{k+1/2} -  \u_h^{k+1/2}] \times \mathbf{H}^{k+1/2}_h, \mathbf{H}^{k+1/2}_h\right)=0.
	\end{align*}
	The term $(D_t \mathbf{A}_h, \mathbf{B}^{k+1/2}_h) $ vanishes with an analogous proof.
\end{proof}

\subsection{Helicity and energy preserving scheme for $\u \cdot \n =  0$}\label{sec:schemeucdotn}
We now consider the boundary conditions $\u \cdot \n =  0$ on $\partial \Omega$. The presented scheme preserves the energy and magnetic helicity precisely, and in contrast to the previous algorithm also enforces $\nabla \cdot \u_h = 0$ precisely, but it does not preserve the hybrid helicity. Again, we only focus on $\f = \mathbf{0}$ and formally $\Re^{-1} = \Reminv =\infty$. 

The following algorithm is mainly taken from \cite{gawlik2020}, but adapted for the additional Hall-term.  
The semi-discrete form of our algorithm is given by: find $(\u_h(t), p_h(t), \B_h(t), \j_h(t)) \in \Hhd \times  \Ltz \times \Hhd \times \Hhc$ such that

\begin{subequations}\label{alg:helicity-operator3}
	\begin{align}
		((\u_h)_t, \v_h) + (\Qc[(\tilde \nabla_h \times \u_h) \times \Qc \u_h ], \v_h) - S (\Qc[\j_h \times \Qc \B_h] ,\v_h) - (p_h, \nabla \cdot \v_h) = 0 &\quad \forall\, \v_h \in \Hhd,  \label{alg:helicity-operator3-u}\\
		(\nabla \cdot \u_h, q_h) = 0 &\quad \forall\, q_h \in \Ltz, \\
		((\B_h)_t, \C_h) - (\nabla \times \Qc[\Qc \u_h \times \Qc \B_h], \C_h ) + \RH (\nabla \times \Qc[\j_h \times \Qc \B_h], \C_h ) = 0    &\quad \forall\, \C_h \in \Hhd, \label{alg:helicity-operator3-b}\\
		(\j_h, \k_h) - (\B_h, \nabla \times \k_h) = 0&\quad \forall\, \k_h \in \Hhc.
	\end{align}
\end{subequations}

For the following theorems, we only show the part of the proof that involves the additional Hall-term. The remainders of the proofs then coincide with the ones in \cite{gawlik2020}.

\begin{remark}
	Similar to before, every solution satisfies $\nabla \cdot \B_h = 0$ if $\nabla \cdot \B_h^0=0$. Furthermore, the $\Hd$-$L^2(\Omega)$ discretization allows the exact enforcement of $\nabla \cdot \u_h=0$, e.g., for $\V_h = \mathbb{BDM}_k$ or $\V_h = \mathbb{RT}_k$ and $Q_h = \mathbb{DG}_{k-1}$ since then $\nabla \cdot \V_h \subset Q_h$.
\end{remark}

\begin{theorem}\label{thm:energyudotn}
	Any solution $(\u_{h}, p_h, \B_{h}, \j_{h})$ of \eqref{alg:helicity-operator3} satisfies the energy identity
	$$
	\frac{1}{2}\frac{d}{dt}(\|\u_{h}\|^{2}+S\|\B_{h}\|^{2})= 0.
	$$
\end{theorem}
\begin{proof}
For the energy identity, it is crucial that the additional Hall term vanishes when \eqref{alg:helicity-operator3-b} is tested with $\B_h$. Indeed, we have that 
\begin{equation}
	\RH\,(\nabla \times \Qc [\j_h \times \Qc \B_h], \B_h) = \RH\,(\Qc[\j_h \times \Qc \B_h], \j_h) = 0
\end{equation}
since $\Qc \j_h=\j_h$ for $\j_h \in \Hhc$.
\end{proof}

\begin{theorem}\label{thm:magheludotn}
	The magnetic helicity of \eqref{alg:helicity-operator3} is conserved if $\f=\mathbf{0}$ and formally $\Re^{-1}=\Reminv =0$.
\end{theorem}

\begin{proof}
	We have to show that the Hall-term vanishes when \eqref{alg:helicity-operator3-b} is tested with a vector-potential $\mathbf{A}_h$. Calculating,
\begin{equation}
	  \RH (\nabla \times \Qc[\j_h \times \Qc \B_h],  \mathbf{A}_h ) =  \RH (\Qc[\j_h \times \Qc \B_h],  \mathbf{B}_h ) =\RH (\j_h \times \Qc \B_h,  \Qc \mathbf{B}_h ) = 0.
\end{equation}
\end{proof}

\begin{remark}
	We discuss why a scheme that conserves hybrid helicity is difficult to construct for the boundary conditions $\u \cdot \n = 0$. First, these boundary conditions naturally fit with $\u_h \in \Hhd$. Therefore, the definition of the discrete hybrid helicity is not straight-forward due to the term $\nabla \times \u$. Two possible choices could be
	\begin{equation}
	    H_{H}:=\int_\Omega (\mathbf{A}_h + \alpha \u_h)\cdot (\B_h + \beta \bm \omega_h) \, \mathrm{d}x,
	\end{equation}
with either $\bm \omega_h = \nabla \times \Qc \u_h$ or $\bm \omega_h = \tilde{\nabla}_h \times \u_h$. The evolution of the fluid helicity would coincide for both definitions since
\begin{equation}
	\frac{d}{dt} (\u_h,  \nabla \times \Qc \u_h) = ((\u_h)_t,  \nabla \times \Qc \u_h) + (\u_h,  \nabla \times \Qc (\u_h)_t) = ((\u_h)_t, \tilde{\nabla}_h \times \u_h  + \nabla \times \Qc \u_h  )
\end{equation}
and 
\begin{equation}
	\frac{d}{dt} (\u_h,  \tilde{\nabla}_h \times \u_h ) = ((\u_h)_t,  \tilde{\nabla}_h \times \u_h ) + (\u_h,  \tilde{\nabla}_h \times (\u_h)_t ) = ((\u_h)_t, \nabla \times \Qc \u_h + \tilde{\nabla}_h \times \u_h ).
\end{equation}
The right-hand side can be modified to $ \nabla \times \Qc \u_h + \Qd^0\tilde{\nabla}_h \times \u_h$, where $\Qd^0$ denotes the projection to the divergence-free functions in $\Hhd$. This ensures that this term is a suitable test function in the velocity equation and that the term $(p_h, \nabla \cdot \v_h)$ vanishes.

An essential step in a proof for the hybrid helicity conservation on the continuous level is that the advection term from the Navier--Stokes equations vanishes when tested against $\bm \omega$, i.e., 	$(\u \times \bm \omega, \bm \omega) = 0$. This already requires a complicated discretization of the advection term. A possible choice could be to approximate $u\times \bm \omega$ by
\begin{equation}
	\frac{1}{2}\Qd[\u \times [ \nabla \times \Qc \u_h + \Qd^0\tilde{\nabla}_h \times \u_h]].
\end{equation}
However, the essence of the conservation proofs is the cancellation of corresponding terms that result from the symmetry in the discretization. That means also the Lorentz force, the Hall-term and magnetic advection terms have to be discretized in a similar complicated way. The authors were not able to find an elegant discretization that does not require the introduction of many additional terms and auxiliary variables.
\end{remark}

Again, to render \eqref{alg:helicity-operator3} computable we introduce auxiliary variables for the projections, yielding: find $(\u_{h}(t), p_h(t), \B_{h}(t), \E_{h}(t), \j_{h}(t), \mathbf{H}_{h}(t),\bm \omega_{h}(t), \mathbf{U}_h(t), \bm \alpha_h(t)) \in$ $\Hhd \times L^2_0(\Omega) \times \Hhd \times [\Hhc]^6 ) $, such that for any $(\v_h, q_h, \C_h, \F_h, \k_h, \bm G_h, \bm \mu_h, \mathbf{V}_h, \bm \beta_h)$ in the same space,
\begin{subequations}\label{alg:helicityudotn}
	\begin{align}\label{alg:helicity-u}
		((\u_{h})_{t}, \v_h)+ (\bm \alpha_h, \v_h), \v_h)+( \nabla \cdot \v_h,  p_{h})&= 0,
		\\
		( \nabla \cdot \u_{h}, q_{h})&=0,
		\\\label{alg:helicity-b}
		((\B_{h})_{t}, \C_h)+(\nabla\times  \E_{h}, \C_h)&=0,\\\label{alg:helicity-j}
		(\j_{h}, \F_h)-(\B_{h},  \nabla\times \F_h)&=0,\\\label{alg:helicity-h1}
		(\bm H_{h}, \bm G_h)-(\B_{h},  \bm G_h)&=0,\\\label{alg:helicity-h2}
		(\bm \omega_{h}, \bm \mu_h)-(\u_{h},  \nabla\times \bm \mu_h)&=0,\\\label{alg:helicity-U2}
		(\mathbf{U}_{h}, \mathbf{V}_h)-(\u_{h},  \mathbf{V}_h)&=0,\\\label{alg:helicity-alpha2}
		(\bm \alpha_{h}, \bm \beta_h)+(\bm \omega_h \times \mathbf{U}_h,  \bm \beta_h)- S (\j_h \times \mathbf{H}_h,  \bm \beta_h)&=0,\\\label{alg:helicity-omega2}
		(\E_{h}, \k_h)-((\RH\,\j_{h}-\mathbf{U}_{h})\times \bm H_{h},  \k_h)&=0.
	\end{align}
\end{subequations}
Now \eqref{alg:helicity-j} gives $\j_{h}=\tilde\nabla_{h}\times \B_{h}$,  \eqref{alg:helicity-h1} gives $\bm H_{h}=\Qc \B_{h}$; \eqref{alg:helicity-h2} gives $\bm \omega_{h}=\tilde{\nabla}_h\times \u_{h}$, \eqref{alg:helicity-U2} gives $\mathbf{U}_{h}=\Qc \u_{h}$ and \eqref{alg:helicity-alpha2} gives $\bm \alpha_{h}= \Qc[(\tilde{\nabla}_h \times \u_h) \times \Qc \u_h] - S \Qc[\j_h \times \Qc \B_h]$.

We use the same time discretization as in Section \ref{sec:schemeutimesn}; compare also to \cite[Section 6]{gawlik2020} for a detailed proof of the next theorem. The proofs for the Hall-term follow immediately from the continuous proofs of Theorem \ref{thm:energyudotn} and Theorem \ref{thm:magheludotn}.

\begin{theorem}
	The time-discretized version of \eqref{alg:helicityudotn} preserves the energy and magnetic helicity precisely and enforces $\div \B_h=\div \u_h=0$  for all time steps; i.e. for all $ k \geq 0$ there holds
	
	\begin{align}
		\int_\Omega \mathbf{u}^{k+1}_h  \cdot \u_h^{k+1} + S \mathbf{B}^{k+1}_h  \cdot \B_h^{k+1} \mathrm{d} x &= \int_\Omega  \mathbf{u}^{k}_h  \cdot \u_h^{k} + S \mathbf{B}^{k}_h  \cdot \B_h^{k} \mathrm{d} x,\\	
		\int_\Omega \mathbf{A}^{k+1}_h  \cdot \B_h^{k+1} \mathrm{d} x &= \int_\Omega \mathbf{A}^{k}_h  \cdot \B_h^{k} \mathrm{d} x, \label{eq:magheltime}\\
		\div \u^k_h = 0,\\
		\div \B^k_h = 0.
	\end{align}
\end{theorem}

\section{Augmented Lagrangian preconditioner}\label{sec:ALP}
In this section, we derive block preconditioners for the stationary and time-dependent versions of the Picard and Newton linearizations from  Algorithm \ref{alg:picard-s} and  Algorithm \ref{alg:newton-s}. In each nonlinear step, we have to solve a linear system of the form

\begin{equation}
	\label{eq:matrix_upBE}
	\begin{bmatrix}
		\mathcal{F} & \BB^\top &  \mathbf{0} & \tilde{\KK} & \KK\\
		\BB & \mathbf{0} & \mathbf{0} & \mathbf{0} & \mathbf{0} \\
				\mathbf{0}& \mathbf{0}& \mathbf{0} & -\AA & \MM \\
		\mathbf{0} & \mathbf{0} & \DD  & \CC & \mathbf{0}\\
		-\GG& \mathbf{0}& -\mathcal{P} & -\tilde{\GG} + \tilde{\NN}  & \LL + \NN
	\end{bmatrix}
	\begin{bmatrix}
		x_{\u_h} \\ x_{p_h} \\ x_{\E_h} \\ x_{\B_h} \\ x_{\j_h}
	\end{bmatrix} =
	\begin{bmatrix}
		R_{\u_h} \\ R_{p_h} \\ R_{\E_h} \\ R_{\B_h} \\ R_{\j_h}
	\end{bmatrix},
\end{equation}
where $x_{\u_h}$, $x_{p_h}$, $x_{\E_h}$, $x_{\B_h}$ and $x_{\j_h}$ are the coefficients of the discretized Newton corrections and $R_{\u_h}$, $R_{p_h}$, $R_{\E_h}$, $R_{\B_h}$ and $R_{\j_h}$ the corresponding nonlinear residuals.
The correspondence between the discrete and continuous operators is illustrated in Table \ref{tab:Operators}. We have chosen the notation that operators that include a tilde are omitted in the Picard linearization from Algorithm $\ref{alg:picard-s}$. In the time-dependent case, the terms ${(\Delta t})^{-1} \u^n_h$ and $({\Delta t})^{-1} \B^n_h$ are added to $\mathcal{F}$ and $\mathcal{C}$, respectively.

\begin{table}[htb!]
	\centering
	\begin{tabular}{c|c|c}
		\toprule
		\textbf{Discrete} & \textbf{Continuous} & \textbf{Weak form}\\
		\midrule
		$\mathcal{F} \u^n_h$ & $-\frac{1}{\Re} \Delta \u_h^n + \u_h^{n-1}\cdot \nabla \u_h^n + \u_h^n \cdot \nabla \u_h^{n-1}$ & $\frac{1}{\Re}(\nabla\u_h^n, \nabla \v_h) + (\u_h^{n-1}\cdot \nabla \u_h^{n} , \v_h)$  \\
		& $-\gamma\nabla \nabla \cdot \u_h^n$ &  $+(\u_h^n\cdot \nabla \u_h^{n-1}, \v_h) + \gamma(\nabla \cdot \u_h^n, \nabla \cdot \v_h)  $  \\
		$\KK \j^n_h$& $-S\, \j^n_h \times \B^{n-1}_h$ & $ -S\, (\j^n_h \times \B^{n-1}_h, \v_h)$ \\
		$\tilde{\KK} \B^n_h$& $-S\, \j^{n-1}_h \times \B^{n}_h$ & $ -S\, (\j^{n-1}_h \times \B^{n}_h, \v_h)$ \\
		$\BB^\top p^n_h$ & $\nabla p^n_h$ & $-(p^n_h, \div \v_h)$ \\
		$\BB \u^n_h$ & $-\div \u^n_h$ & $-(\div \u^n_h, q)$  \\
		$\LL \j^n_h$ & $ \frac{1}{\Rem} \j^n_h$ & $\frac{1}{\Rem} (\j^n_h, \k_h)$ \\
		$\mathcal{P} \E^n_h$ & $\E^n_h$ & $ (\E^n_h, \k_h)$ \\
		$\GG \u^n_h$ & $\u^n_h \times \B^{n-1}_h$ & $(\u^n_h\times\B^{n-1}_h, \k_h)$  \\
		$\tilde{\GG} \B^{n}_h$ & $\u^{n-1}_h \times \B^{n}_h$ & $(\u^{n-1}_h\times\B^{n}_h, \k_h)$  \\
		$\NN \j^n_h$ & $\RH\, \j^n_h \times \B^{n-1}_h$ & $\RH\,(\j^n_h\times\B^{n-1}_h, \k_h)$  \\
		$\tilde{\NN} \B^{n}_h$ & $\RH\, \j^{n-1}_h \times \B^{n}_h$ & $\RH\,(\j^{n-1}_h\times\B^{n}_h, \k_h)$  \\
		$\DD \E^n_h$ & $ \nabla \times \E^n_h$ & $ (\nabla \times \E^n_h, \C_h)$ \\
		$\CC \B^n_h$ & $-\nabla \nabla \cdot \B^n_h$ & $(\nabla \cdot \B^n_h,\nabla \cdot \C_h)$  \\
		$\MM \j^n_h$ & $\j^n_h$ & $ (\j^n_h, \F_h)$ \\
		$\A \B^n_h$ & $ \nabla \times \B^n_h$ & $ (\B^n_h, \nabla \times  \F_h)$ \\
		\bottomrule
	\end{tabular}
	\caption{Overview of operators. For implicit Euler, the terms ${(\Delta t})^{-1} \u^n_h$ and $({\Delta t})^{-1} \B^n_h$ are added to $\mathcal{F}$ and $\mathcal{C}$, respectively.}
	\label{tab:Operators}
\end{table}
The following preconditioning approach is similar to one developed in \cite{laakmann2021} for the standard incompressible resistive MHD equations. The main idea is to do a Schur complement approximation which separates the hydrodynamic and electromagnetic unknowns and then to apply parameter-robust multigrid methods to the different subproblems.

We start by simplifying the outer Schur complement that eliminates the $(\u_h, p_h)$ block given by
\begin{equation}\label{eq:outerSchurCompup}
	\mathcal{S}^{(\u_h, p_h)} =
	\begin{bmatrix}
			\mathbf{0} & -\AA & \MM\\
				\DD & \CC & \mathbf{0}\\
		-\mathcal{P}& -\tilde{\GG} + \tilde{\NN} & \LL + \NN
	\end{bmatrix}
	-
	\begin{bmatrix}
		\zerov & \zerov\\
		\zerov & \zerov\\
		-\GG & \zerov
	\end{bmatrix}
	\begin{bmatrix}
		\FF & \BB^\top \\
		\BB & \mathbf{0}
	\end{bmatrix}^{-1}
	\begin{bmatrix}
		\zerov & \tilde{\KK} & \KK \\
		\zerov & \zerov & \zerov
	\end{bmatrix}.
\end{equation}
Applying the identity
\begin{equation}
	\label{eq:matrixinvers}
	\begin{bmatrix}
		A & B \\
		C & D
	\end{bmatrix}^{-1} = \begin{bmatrix}
		A^{-1} + A^{-1} B (D - C A^{-1} B)^{-1} C A^{-1} &
		-A^{-1}B(D-C A^{-1} B)^{-1} \\
		- (D - C A^{-1} B)^{-1} C A^{-1}&
		(D-C A^{-1}B)^{-1}
	\end{bmatrix}
\end{equation}
for non-singular matrices $A$ and $D-C A^{-1} B$ to the $(\u_h, p_h)$ block results in
\begin{equation}\label{eq:outerSchurCompup2}
	\mathcal{S}^{(\u_h, p_h)} =
	\begin{bmatrix}
 		\mathbf{0} & -\AA & \MM \\
		\DD & \CC & \mathbf{0} \\
				-\mathcal{P}& -\tilde{\GG} + \tilde{\NN} + \GG \mathcal{S}^{-1}_{1,1} \tilde{K} & \LL + \NN + \GG \mathcal{S}^{-1}_{1,1} \KK 
	\end{bmatrix}
\end{equation}
with 
\begin{equation}
	\mathcal{S}^{-1}_{1,1} = \FF^{-1} - \FF^{-1}  \BB^\top (\BB \FF^{-1} \BB^T)^{-1} \BB \FF^{-1}.
\end{equation}
Note that the magnitude of the matrices $\GG \mathcal{S}^{-1}_{1,1} \tilde{K}$ and $\GG \mathcal{S}^{-1}_{1,1} \KK $ is approximately a factor of $\mathcal{O}(h^2)$ smaller than of the other matrices at the corresponding entries. Therefore, a good approximation for a reasonably refined mesh and moderate coupling numbers $S$ is given by
\begin{equation}\label{eq:outerSchurCompupapprox}
	\mathcal{S}^{(\u_h, p_h)}_\text{approx} =
	\begin{bmatrix}
 			\mathbf{0} & -\AA & \MM\\
 \DD & \CC & \mathbf{0}\\
 -\mathcal{P}& -\tilde{\GG} + \tilde{\NN} & \LL + \NN
	\end{bmatrix}.
\end{equation}

The treatment of the hydrodynamic block 
\begin{equation}\label{eq:Schurcomphycdro}
        \mathcal{M}_{NS}=	
		\begin{bmatrix}
			\FF & \BB^\top \\
			\BB & \mathbf{0}
		\end{bmatrix}
\end{equation}
coincides with the one described in \cite[Section 3.4]{laakmann2021}. Therefore, we also add the augmented Lagrangian term $\gamma (\nabla \cdot \u^n_h, \nabla \cdot \v^n_h)$ to the velocity equation with a large $\gamma$ to gain control over the Schur complement of \eqref{eq:Schurcomphycdro}. Moreover, we use an $\Hd$-conforming discretization of $\u_h$ \cite[Section 2.3]{laakmann2021} to allow the use of parameter-robust multigrid methods that can deal with the non-trivial kernels of the occurring semi-definite terms; for more information about this topic we refer to \cite{schoberl1999b}.

We found that applying the same parameter-robust multigrid methods monolithically to the Schur complement approximation $\tilde{\mathcal{S}}^{(\u_h, p_h)}$ shows good results for the three dimensional lid-driven cavity problem as long as $\Rem$, $\S$ and $\RH$ are not chosen too high at the same time. 

For completeness, we also outline the block structure of the 2.5D formulation introduced in Section \ref{sec:2.5Dform}. We use $\eta \in \{0,1\}$ to distinguish between the stationary $(\eta=0)$ and transient $(\eta=1)$ cases. The hydrodynamic block 	$\begin{bmatrix}
	\FF & \BB^\top \\
	\BB & \mathbf{0}
\end{bmatrix}$ arises now as the discretization of the forms 
\begin{equation} 
	\begin{bmatrix}
		A_1 & \mathbf{0} & (p_h, \nabla \cdot \tilde{\v}_h)\\
		\mathbf{0} & A_2 &  \mathbf{0} \\
		(\nabla \cdot \tilde{\u}_h, q_h) &  \mathbf{0} &  \mathbf{0}
	\end{bmatrix}
\end{equation}
with
\begin{align*}
	A_1 &= 	\frac{\eta}{\Delta t} (\tilde{\u}^n_h, \tilde{\v}_h) -\frac{1}{\Re} (\nabla \tilde{\u}^n_h, \nabla \tilde{\v}_h) + 	( (\tilde{\u}^n_h \cdot \tilde{\nabla}) \tilde{\u}^{n-1}_h, \tilde{\v}_h) + ((\tilde{\u}^{n-1}_h \cdot \tilde{\nabla}) \tilde{\u}^{n}_h, \tilde{\v}_h)\\
	&  + \gamma (\nabla \cdot \tilde{\u}^n_h, \nabla \cdot \tilde{\v}^n_h), \\
	A_2 &= \frac{\eta}{\Delta t} (u^n_{3h}, v_{3h}) -\frac{1}{\Re} (\nabla u^n_{3h}, \nabla v_{3h}) + 	( (\tilde{\u}^n_h \cdot \tilde{\nabla}) u_{3h}^{n-1}, v_{3h}) + ((\tilde{\u}^{n-1}_h \cdot \tilde{\nabla}) u^n_{3h}, v_{3h}).
\end{align*}

Furthermore,
\begin{equation}
	\begin{bmatrix}
		\mathbf{0} & \tilde{\KK} & \KK \\
		\mathbf{0} & \mathbf{0} & \mathbf{0} \\
	\end{bmatrix}, 
	\begin{bmatrix}
		\mathbf{0} & \mathbf{0} \\
		\mathbf{0} & \mathbf{0} \\
		-\GG & \mathbf{0}
    \end{bmatrix}
    \text{ and }
    \begin{bmatrix}
    	\mathbf{0} & -\AA & \MM \\
    	\DD & \CC & \mathbf{0} \\
    	-\mathcal{P}& -\tilde{\GG} + \tilde{\NN} & \LL + \NN 
    \end{bmatrix}
\end{equation}
correspond to 
\begin{equation} 
	\resizebox{\textwidth}{!}{$
		\begin{bmatrix}
			\mathbf{0}  &  \mathbf{0}  & S(\tilde{\B}^n_h\times j^{n-1}_{3h}, \tilde{\v}_h) &  -S(\tilde{\j}^{n-1}_h \times \B^{n}_{3h}, \tilde{\v}_h) &  -S (\tilde{\j}^n_h \times \B^{n-1}_{3h}, \tilde{\v}_h)& S(\tilde{\B}^{n-1}_h\times j^{n}_{3h}, \tilde{\v}_h) \\
			\mathbf{0} & \mathbf{0} & -S (\tilde{\j}^{n-1}_{h}\times \tilde{\B}^{n}_h, v_{3h}) & \mathbf{0} &  -S (\tilde{\j}^{n}_{h}\times \tilde{\B}^{n-1}_h, v_{3h}) & \mathbf{0} \\
			\mathbf{0} & \mathbf{0} & \mathbf{0} & \mathbf{0} &  \mathbf{0} & \mathbf{0} \\
		\end{bmatrix}$},
\end{equation}

\begin{equation} 
	\begin{bmatrix}
		\mathbf{0} & \mathbf{0} & \mathbf{0}\\
		\mathbf{0} & \mathbf{0} & \mathbf{0}\\
		\mathbf{0} & \mathbf{0} & \mathbf{0}\\
		-(\tilde{\u}^n_h \times B^{n-1}_{3h}, \k_h) & (\tilde{\B}^{n}_h \times u^{n-1}_{3h}, \k_h) & \mathbf{0}\\
		-(\tilde{\u}^n_h \times \tilde{\B}^{n-1}_h,k_{3h})& \mathbf{0} & \mathbf{0}
	\end{bmatrix}
\end{equation}

and 

\begin{equation}
	\resizebox{\textwidth}{!}{$
		\begin{bmatrix}
			\mathbf{0} & \mathbf{0} &\mathbf{0} & -(B^n_{3h}, \scurl \tilde{\F}_h) &  (\tilde{\j}^n_h, \tilde{\F}_h)&\mathbf{0} \\
			\mathbf{0} & 	\mathbf{0} & -(\tilde{\B}^n_h, \vcurl F_{3h}) &	\mathbf{0} & 	\mathbf{0} &	(j^n_{3h}, F_{3h}) \\
			\mathbf{0} & (\vcurl E^n_{3h} , \tilde{\C}_h) & \substack{\frac{\eta}{\Delta t} (\tilde{\B}^n_h, \tilde{\C}_h)\\+ \frac{1}{\Rem} (\nabla \cdot \tilde{\B}^n_h, \nabla \cdot \tilde{\C}_h)}& \mathbf{0} & \mathbf{0} & \mathbf{0} \\ 
			(\scurl \tilde{\E}^n_h, \C_{3h}) & \mathbf{0}& \mathbf{0}&  \frac{\eta}{\Delta t} (B^n_{3h}, \C_{3h}) &  \mathbf{0}& \mathbf{0}\\
			-(\tilde{\E}^n_h, \k_h) & \mathbf{0} &\substack{(\tilde{\B}^n_h \times u^{n-1}_{3h}, \k_h) \\- \RH (\tilde{\B}^n_h \times j^{n-1}_{3h}, \k_h)} & 
			\substack{(\tilde{\u}^{n-1}_h \times B^n_{3h}, \k_h) \\+ \RH (\tilde{\j}^{n-1}_h \times B^n_{3h},\k_h)} &  \substack{\frac{1}{\Rem} (\tilde{\j}^n_h, \k_h) \\ + \RH (\tilde{\j}^n_h \times B^{n-1}_{3h}, \k_h)} & -\RH (\tilde{\B}^{n-1}_h \times j^n_{3h}, \k_h)  \\ 
			\mathbf{0} & -(E^n_{3h}, k_{3h}) & \substack{- (\tilde{\u}^{n-1}_h \times \tilde{\B}^n_h, k_{3h}) \\+ \RH (\tilde{\j}^n_h  \times \tilde{\B}^{n-1}, k_{3h})} &\mathbf{0}  & \RH(\tilde{\j}^n_h \times \tilde{\B}^{n-1}_h, k_{3h}) & \frac{1}{\Rem} (j^n_{3h}, k_{3h})
		\end{bmatrix}$}.
\end{equation}
Our numerical experiments suggest that the same outer Schur complement approximation (now applied to the blocking $(\tilde{\u}_h, u_{3h}, p_h)$ and $(\tilde{\B}_h,B_{3h},\tilde{\E}_h,E_{3h},\tilde{\j}_h,j_{3h})$) still works well for the 2.5D case. However, we observe poor performance of the monolithic multigrid method applied to this block for an island coalescence and $\RH>0.01$. Robust solvers for this inner problem require further investigation and we apply a direct solver to this block in the 2.5D numerical results in the next section.

\section{Numerical Results}\label{sec:numericalresults}
The following numerical results were implemented in Firedrake \cite{rathgeber2016}, which uses the solver package PETSc \cite{balay2019} and the implementation of parameter-robust multigrid methods from PCPATCH \cite{farrell2019pcpatch}. Moreover, we replaced the Laplace term $-\Delta \u$ in our implementation by $-2\nabla \cdot \varepsilon(\u)$, where $\varepsilon(\u) := 1/2(\nabla \u + \nabla \u^\top)$ denotes the symmetric gradient. This allows us to also consider alternative boundary conditions 
\begin{equation}
	\u = \mathbf{0} \text{ on } \Gamma_D , \qquad \frac{2}{\Re}  \varepsilon( \u )\cdot \n = p \n \text{ on } \Gamma_N
\end{equation}
with $\Gamma_D\cup \Gamma_N = \partial \Omega$.
Note that both formulations are equivalent for the boundary conditions $\u=\mathbf{0}$ on $\partial \Omega$ which we consider in this paper \cite[Chap.~15]{Quarteroni2017}.

\subsection{Verification and convergence order}\label{sec:verificationandconvergenceorder}
In the first example, we consider the method of manufactured solutions for a smooth given solution to verify the implementation of our solver and report convergence rates. We employ the Picard iteration for the stationary problem from Algorithm \ref{alg:picard-s}.
The right-hand sides and boundary conditions are calculated corresponding to the analytical solution
\begin{gather}
	\begin{split}
	\u(x,y,z) = \begin{pmatrix}
       \cos(y) \\ \sin(z) \\ \exp(x)
	\end{pmatrix}, \quad
	p(x,y,z) = y \sin(x) \exp(z), \quad 
	\B(x,y,z)  = \begin{pmatrix}
        \sin(z) \\ \sin(x) \\ \cos(y)
	\end{pmatrix},\\
    \E(x,y,z) = \begin{pmatrix}
	x \sin(x) \\ \exp(y) \\ z^3
    \end{pmatrix}, \quad
    \j(x,y,z) = \begin{pmatrix}
    	\cos(y z) \\ \exp(xz) \\ \sinh(x)
    \end{pmatrix}.  \qquad \qquad \qquad \quad
    \end{split}
\end{gather}
We used second order $\mathbb{BDM}$-elements for $\u_h$, second order $\mathbb{NED}1$-elements for $\E_h$ and $\j_h$, second order $\mathbb{RT}$-elements for $\B_h$ and first order $\mathbb{DG}$-elements for $p_h$ on $\Omega=[0,1]^3$.
Based on the standard error estimates for these spaces, one would expect third order convergence in the $L^2$-norm for $\u_h$ and  second order convergence for $p_h$, $\B_h$, $\E_h$ and $\j_h$. This is numerically verified by Table \ref{tab:convOrder}.
\begin{table}[!htb]
	\begin{centering}
			\resizebox{\textwidth}{!}{
		\begin{tabular}{c |c c|c c|c c|c c|c c} 
			\toprule
			h & $\|\u-\u_h\|_0$ & rate & $\|p-p_h\|_0$ & rate & $\|\B-\B_h\|_0$ & rate & $\|\E-\E_h\|_0$ & rate & $\|\j-\j_h\|_0$ & rate  \\
			\midrule
			1/4 & 3.08E-04 & - & 3.52E-02 & - & 2.44E-03 & - & 9.57E-03 & - & 6.77E-03 & - \\
			1/8 & 4.50E-05 & 2.78 & 6.58E-03 & 2.42 & 6.04E-04 & 2.02 & 2.50E-03 & 1.93 & 1.79E-03 & 1.92 \\
			1/16 & 5.99E-06 & 2.91 & 1.36E-03 & 2.27 & 1.50E-04 & 2.01 & 6.32E-04 & 1.99 & 4.53E-04 & 1.98 \\
			1/32 &7.72E-07 & 2.96 & 2.99E-04 & 2.19 & 3.74E-05 & 2.00 & 1.58E-04 & 2.00 & 1.14E-04 & 1.99 \\
			\bottomrule
		\end{tabular}}
		\caption{$L^2$-error and convergence order}
		\label{tab:convOrder}
	\end{centering}
\end{table}

\subsection{Lid-driven cavity problem}
Next, we consider a lid-driven cavity problem for a background magnetic field $\B_0=(0,0,1)^\top$ which determines the boundary conditions $\B\cdot \n = \B_0\cdot \n$ on $\partial \Omega$ and set $\f = \mathbf{0}$ for $\Omega=(-0.5, 0.5)^3$. The boundary condition $\u=(1,0,0)^\top$ is imposed at the boundary $y=0.5$ and homogeneous boundary conditions elsewhere. The problem models the flow of a conducting fluid driven by the movement of the lid at the top of the cavity. The magnetic field imposed orthogonal to the lid creates a Lorentz force that perturbs the flow of the fluid. 

Since we consider non-homogeneous boundary conditions in this problem the boundary conditions for $\E$ and $\j$ have to be chosen in a compatible way, which we derive in the following. 
From \eqref{eq:HallMHDj} we can deduce the necessary condition that 
\begin{equation}\label{eq:bcscompatibility}
	\Reminv \j \times \n = \E \times \n + (\u \times \B) \times \n - \RH (\j \times \B) \times \n
\end{equation}
has to hold on $\partial \Omega$. 

On a face that does not correspond to $y=0.5$, we have $\u = (0,0,0)^\top$. Then it is clear that \eqref{eq:bcscompatibility} is fulfilled if we choose $\E\times\n=\j\times\n=\mathbf{0}$ on these faces.

On the face $y=0.5$, we have that $\n = (0,1,0)^\top$ and hence \eqref{eq:bcscompatibility} simplifies to
\begin{equation}
	\begin{cases}
		&	\Reminv  j_3 = -E_3 -1 + \RH j_1\\
		&	\Reminv  j_2 = E_2\\
		&	\Reminv  j_1 = E_1 + \RH j_3
	\end{cases}
\end{equation}
If we choose $\E\times\n=\mathbf{0}$ it follows that 
\begin{equation}
	\j \times \n =\frac{1}{\Reminv + \Rem \RH^2} 
	\begin{pmatrix}
		\Rem \RH \\ 0 \\ 1
	\end{pmatrix} 
    \times \n.
\end{equation}

In Table \ref{tab:ldcstatRERHall}, we present iteration numbers for the Picard and Newton linearizations for the stationary version of the lid-driven cavity problem. Here, we have used the same elements for $\u_h$, $\B_h$, $\E_h$ and $\j_h$ and $p_h$ as in the previous example. Moreover, we have used a coarse mesh of $6\times6\times6$ cells and 3 levels of refinement for the multigrid method resulting in an $48\times48\times48$ mesh and 29.2 million degrees of freedom. One can observe good robustness in the reported ranges of $\RH$ for both linearizations. The Newton linearization shows slightly better non-linear convergence, while the linear iterations are slightly smaller in most cases for the Picard iteration.

Table \ref{tab:ldctimeRERHall} shows the corresponding results for the time-dependent version of the lid-driven cavity problem. Here, we have chosen a time step of $\Delta t=0.01$ and iterated until the final time of $T=0.1$. We iterated some of the cases until the final time of $T=1.0$ to confirm that the reported iteration numbers remain representative for longer final times. We have chosen the L-stable BDF2 method for the time-discretization where the first time step was computed by Crank--Nicolson. We observe good robustness in both the nonlinear and linear iteration numbers for this problem.

\begin{table}[htbp!]
	\centering
		\begin{tabular}{r|ccc|ccc}
			\toprule
			& \multicolumn{3}{c|}{Picard} & \multicolumn{3}{c}{Newton} \\
			\midrule
			$\RH\backslash\Re$  &1 &     100 &    1,000 &1 & 100&     1,000  \\
			\midrule
            0.0 & ( 4) 4.8 & ( 4) 5.5 & ( 4)10.0 & ( 3) 6.0 & ( 4) 4.3 & ( 4) 8.8 \\
            0.1 & ( 4) 5.0 & ( 4) 4.8 & ( 4)10.0 & ( 3) 6.0 & ( 4) 4.3 & ( 4) 9.3 \\
            1.0 & ( 4) 5.3 & ( 4) 4.5 & ( 5)10.2 & ( 3) 5.0 & ( 4) 4.3& ( 4) 12.0 \\

			\bottomrule	
		\end{tabular}
	\caption{Iteration counts for the stationary lid-driven cavity problem. The entries of the table correspond to: (Number of nonlinear iterations) Average number of linear iterations per nonlinear step.\label{tab:ldcstatRERHall}\newline}
	
		\begin{tabular}{r|ccc|ccc}
			\toprule
			& \multicolumn{3}{c|}{Picard} & \multicolumn{3}{c}{Newton} \\
			\midrule
			$\RH\backslash\Re$  &1 &     1,000 &    10,000 & 1 & 1,000 & 10,000 \\
			\midrule
            0.0 & (3.0) 5.6 & (3.1) 2.2 & (3.2) 2.0 & (2.1) 7.5 & (3.1)  2.2 & (3.2)  2.0 \\ 
            0.1 & (3.0) 5.6 & (3.1) 2.2 & (3.2) 2.0 & (2.1) 7.5 & (3.1)  2.2 & (3.2) 2.0 \\ 
            1.0 & (3.0) 5.8 & (3.1) 2.2 & (3.2) 2.0 & (2.2) 7.3 & (3.1)  2.2 & (3.2) 2.0\\ 

			\bottomrule	
	\end{tabular}

	\caption{Iteration counts for the time-dependent lid-driven cavity problem.\label{tab:ldctimeRERHall}}

\end{table}

\newcolumntype{C}{ >{\centering\arraybackslash} m{3.0cm} }
\newcolumntype{D}{ >{\centering\arraybackslash} m{2cm} }

\begin{figure}[htbp!]
	\centering
	\begin{tabular}{|D |C |C |C |C |}
		\hline
	$\Rem = 10$ &	\includegraphics[width=2.5cm]{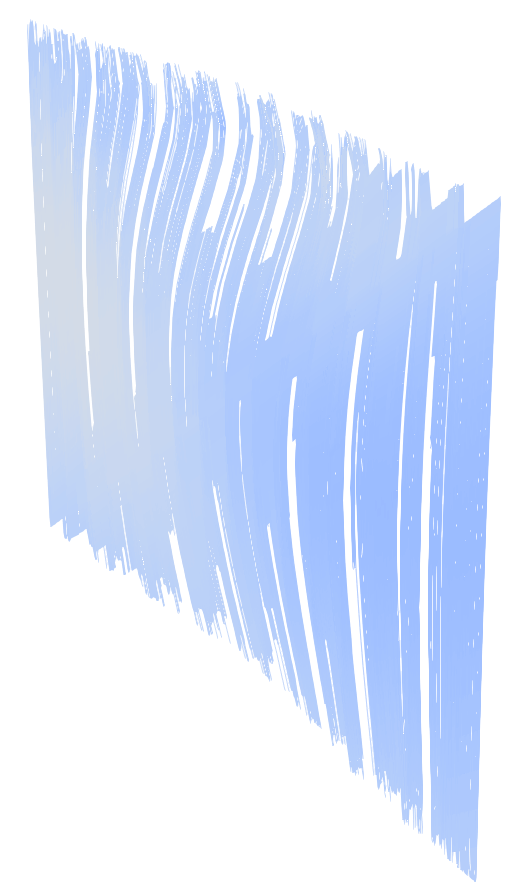} &
		\includegraphics[width=2.5cm]{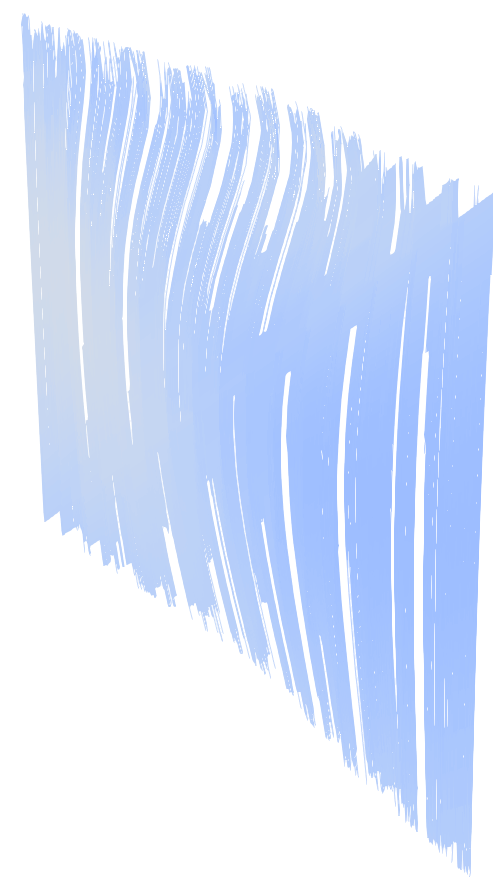} &
		\includegraphics[width=2.5cm]{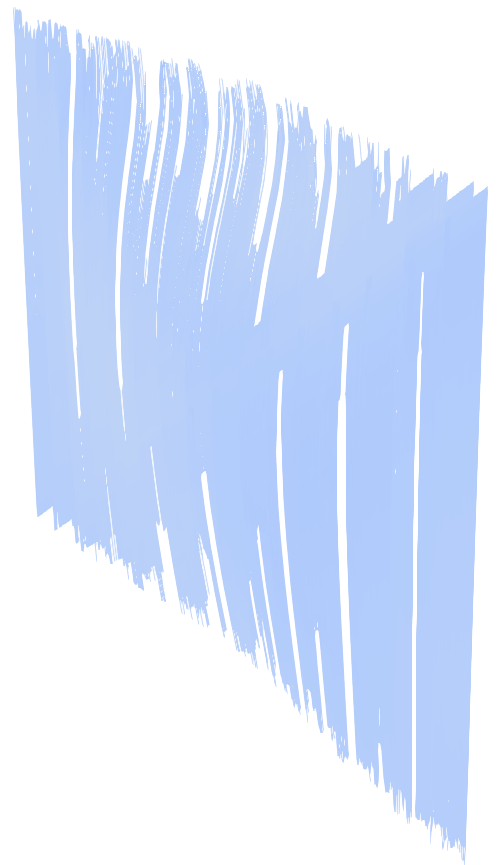} &
		\includegraphics[width=2.5cm]{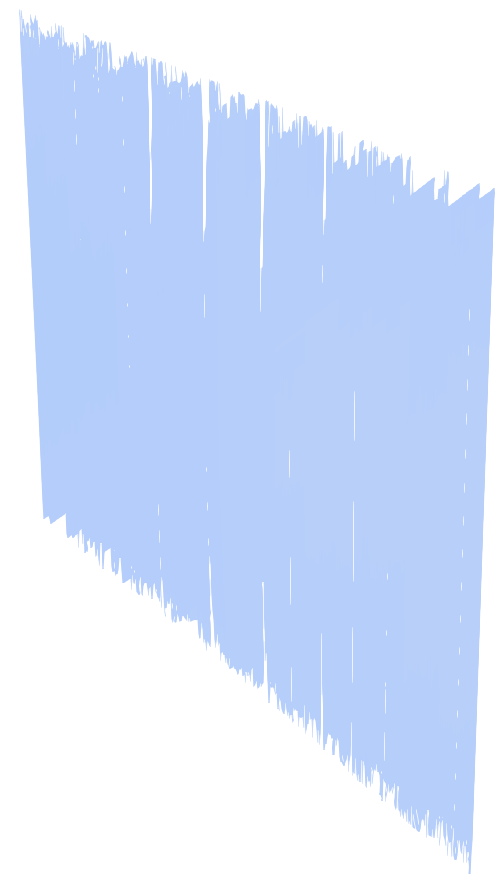} \\
		\hline
	$\Rem = 50$ &	\includegraphics[width=2.5cm]{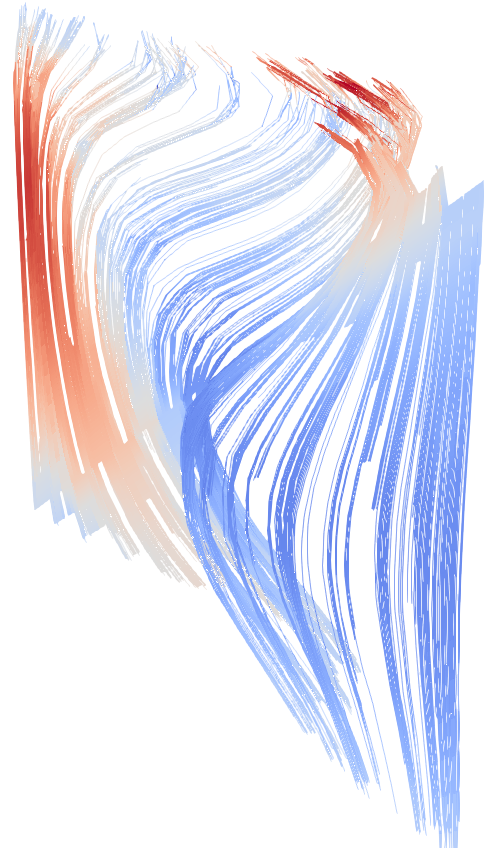} &
		\includegraphics[width=2.5cm]{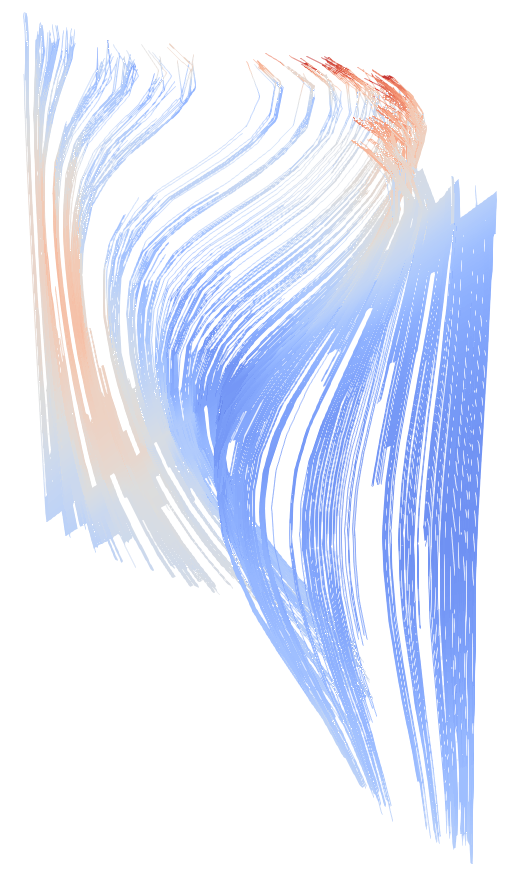} &
		\includegraphics[width=2.5cm]{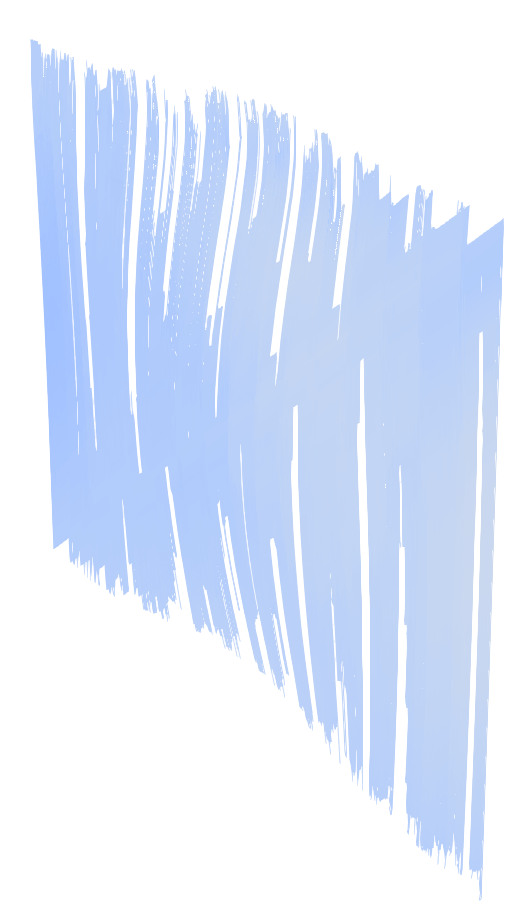} &
		\includegraphics[width=2.5cm]{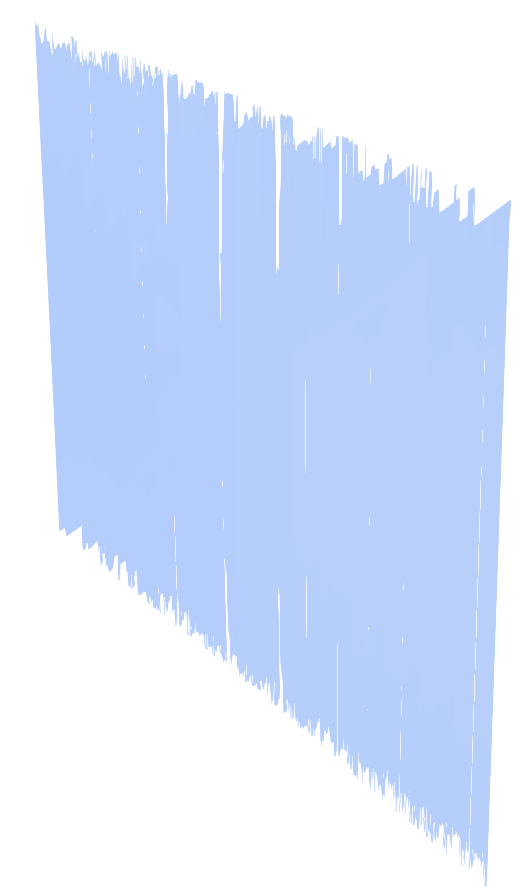} \\
		\hline
	$\Rem = 100$ &	\includegraphics[width=2.5cm]{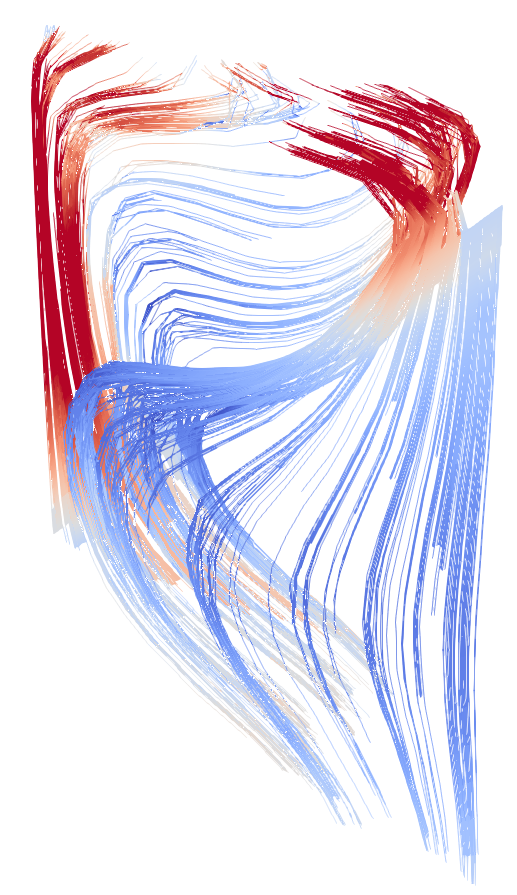} &
		\includegraphics[width=2.5cm]{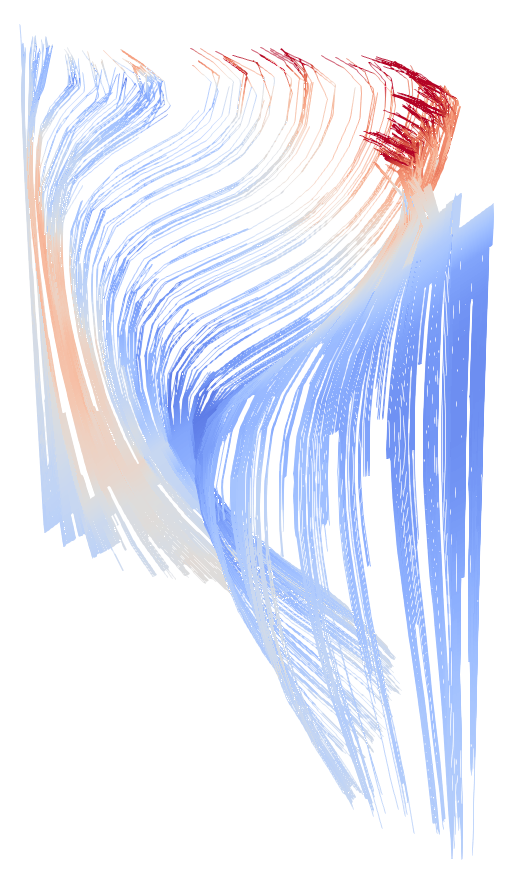} &
		\includegraphics[width=2.5cm]{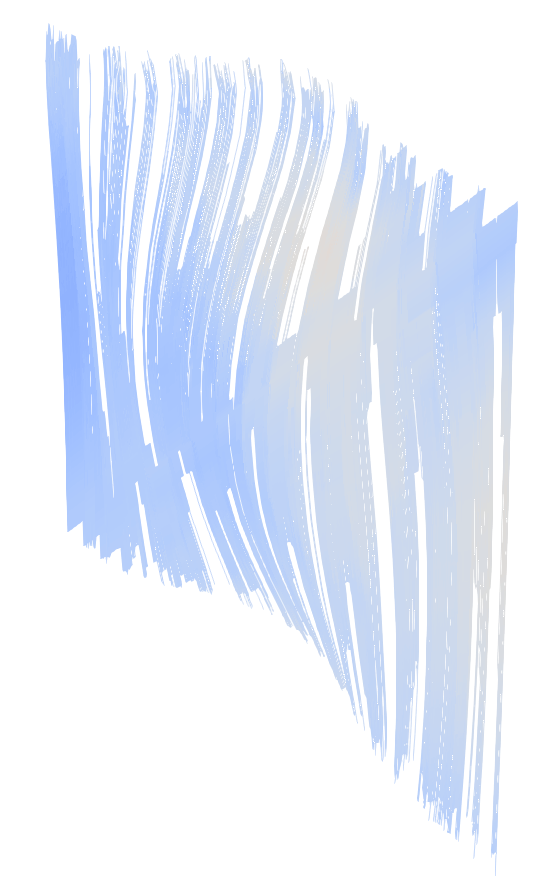} &
		\includegraphics[width=2.5cm]{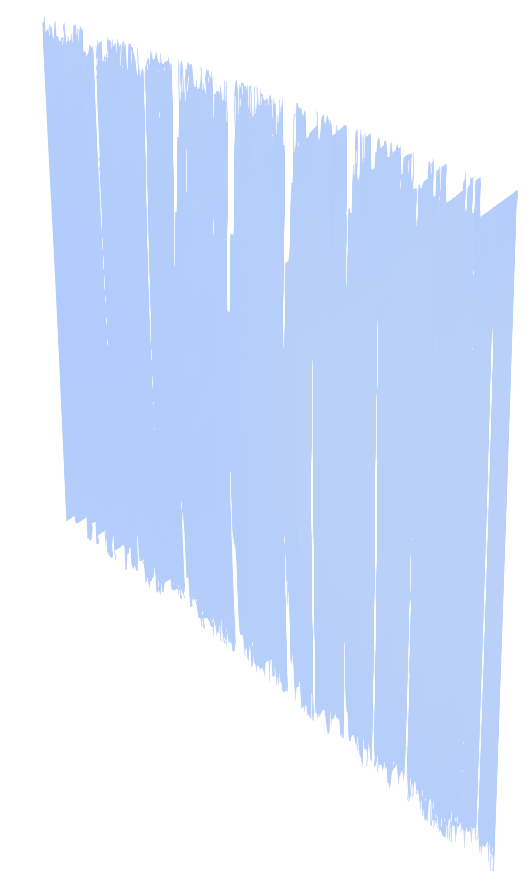} \\
		\hline
       &  $\RH = 0.0$ & $\RH = 0.01$ & $\RH = 0.1$ & $ \RH = 1.0$\\
       \hline
	\end{tabular}
	\caption{Streamlines of the magnetic field for the stationary lid-driven cavity problem for different values of $\Rem$ and $\RH$.\label{fig:Streamlinesldc}}
\end{figure}

Figure \ref{fig:Streamlinesldc} shows plots of the magnetic field for different values of $\Rem$ and $\RH$. For $\RH=0$ one can nicely observe the physical phenomenon that for the standard MHD equations the magnetic fields lines tend to be advected by the fluid flow the higher $\Rem$ is chosen. For increasing $\RH$ one can see that this effect is damped until for $\RH=1$, where the influence of the fluid flow is negligible and the magnetic field is close to the background magnetic field in the direction of $(0,1,0)^\top$. 

\subsection{Test of conservative scheme for $\u \times\n = \mathbf{0}$ }
In this section, we want to numerically verify our results from Section \ref{sec:schemeutimesn} for the boundary conditions $\u \times \n = \mathbf{0}$. Here, we used $\Omega=[0,1]^3$ and a mesh of $12\times12\times12$ cells. We chose the interpolant of the following functions as the initial conditions 
\begin{gather}
	\begin{split}
		\u^0(x,y,z) = \begin{pmatrix}
			-\sin(\pi(x-0.5))\cos(\pi(y-0.5))z(z-1) \\ 
			\cos(\pi(x-0.5))\sin(\pi(y-0.5))z(z-1) \\ 
			0
		\end{pmatrix}, \qquad 
		\B^0(x,y,z)  = \begin{pmatrix}
			-\sin(\pi x)\cos(\pi y) \\ 
			\cos(\pi x) \sin(\pi y) \\ 
			0
		\end{pmatrix},\\
		 \qquad \qquad \qquad
	\end{split}
\end{gather}
which satisfy the boundary conditions $\u^0\times \n = \mathbf{0}$,  $\B^0\times\n=\mathbf{0}$ and the constraints \mbox{$\nabla \cdot \u^0=\nabla \cdot \B^0=0$}. Note that the interpolant of divergence-free functions is still divergence-free for $\mathbb{RT}$ and $\mathbb{BDM}$ elements \cite[Prop. 2.5.2]{boffi2013mixed}. We enforce this property in our implementation by using a sufficiently high quadrature degree in the evaluation of the degrees of freedom for the $\mathbb{RT}$ and $\mathbb{BDM}$ elements; see \cite[Sec. 4.2]{laakmann2021} for more details. Here, we discretize $\u$ with $\mathbb{NED}1$-elements of first order and $p$ with $\mathbb{CG}_1$-elements. 

For the computation of the magnetic helicity we determine a discrete vector-potential such that $\nabla \times \mathbf{A}_h = \B_h$ by the system
\begin{equation}
	\left(\nabla \times \mathbf{A}_h, \nabla \times \k_h\right) = \left(\B_h, \nabla \times \k_h \right) \quad \forall\ \k_h \in \Hhc.
\end{equation}
We solve this singular system with GMRES preconditioned by ILU, which is known to be convergent if the problem is consistent \cite{Ipsen1998}. 

Although, the scheme \eqref{alg:helicityutimesn} contains multiple auxiliary variables, it can be solved efficiently with a fixed point iteration \cite[Section 4]{hu2021helicity}. For the  time step from $t_k$ to $t_{k+1}$ we compute iterative solutions $\left( \u^{(k+1,j)}_{h}, P^{(k+\frac{1}{2},j)}_h,  \B^{(k+1,j)}_{h}, \E^{(k+\frac{1}{2},j)}_{h}, \j^{(k+\frac{1}{2},j)}_{h}, \mathbf{H}^{(k+\frac{1}{2},j)}_{h},\bm \omega^{(k+\frac{1}{2},j)}_{h}\right)$ until the stopping criterion
\begin{equation}
	\frac{\|\u^{(k+1,j+1)}_{h} - \u^{(k+1,j)}_{h}\|}{\|\u^{(k+1,j)}_{h}\|} + \frac{\|\B^{(k+1,j+1)}_{h} - \B^{(k+1,j)}_{h}\|}{\|\B^{(k+1,j)}_{h}\|} < \text{TOL}
\end{equation}
is satisfied for a given tolerance $\text{TOL}$. We initialize the iteration with the values from time step $k$ and first determine the updates  $\left (\E^{(k+\frac{1}{2},j+1)}_{h}, \j^{(k+\frac{1}{2},j+1)}_{h},\mathbf{H}^{(k+\frac{1}{2},j+1)}_{h},\bm \omega^{(k+\frac{1}{2},j+1)}_{h}\right)$ by solving \eqref{alg:helicity-cross-j2} - \eqref{alg:helicity-cross-w2}
with right-hand sides of the level $j$. Then, we update the velocity and pressure by
\begin{alignat}{2}
	\frac{1}{\Delta t} (\u^{(k+1,j+1)}_{h}, \v_h) + (\nabla P^{(k+\frac{1}{2},j+1)}_h, \v_h) &= (\F_h, \v_h)  & \quad \forall\  \v_h \in \Hhd,\\
	(\nabla Q_h, \u^{(k+1,j+1)}_h) &= 0 &\quad \forall \ Q_h \in H^1_0(\Omega), 
\end{alignat}
with
\begin{equation}
	\F_h = \frac{1}{\Delta t} \u^{k}_h + S \j^{(k+\frac{1}{2},j+1)}_h \times \mathbf{H}^{(k+\frac{1}{2},j+1)}_h + \frac{1}{2}\left(\u^{(k+1,j)}_h + \u^{k}_h\right)\times \bm \omega^{(k+\frac{1}{2},j+1)}_h.
\end{equation}
The magnetic field is updated by solving
\begin{equation}
	\frac{1}{\Delta t} (\B^{(k+1,j+1)}_{h}, \C_h) = \frac{1}{\Delta t} (\B^{k}_{h}, \C_h)- (\nabla \times \E^{(k+\frac{1}{2},j+1)}_h, \C_h)  \quad \forall\  \C_h \in \Hhd.
\end{equation}

Figure \ref{fig:uHcurl} shows plots of the different conserved quantities for $\Re=\Rem=\infty$ and $\RH=0.5$. One can clearly see that the energy and hybrid helicity remain constant over time, while the cross and fluid helicity are not conserved. These are the observations we expected from the theory in Section \ref{sec:schemeutimesn}. Moreover, $\div(\B_h)$ and the magnetic helicity also show good preservation with small oscillations on the machine precision level.

In Figure \ref{fig:differentRe}, we show plots of the energy and hybrid helicity for $\RH=0.1$ and multiple finite values of $\Re$ and $\Rem$. This test confirms that both quantities are indeed only conserved in the ideal limit of $\Re=\Rem=\infty$.

Finally, Figure \ref{fig:differentRH} compares the cross and hybrid helicity for different values of $\RH$ in the ideal limit of $\Re=\Rem=\infty$. One can observe that the cross helicity is indeed only conserved for $\RH=0$, which corresponds to the standard MHD equations. On the other hand, the hybrid helicity is conserved for all tested values of $\RH$. Note that the hybrid helicity corresponds for $\RH=0$ to the magnetic helicity.

\begin{figure}[htbp!]
	\centering
		\includegraphics[width=12cm]{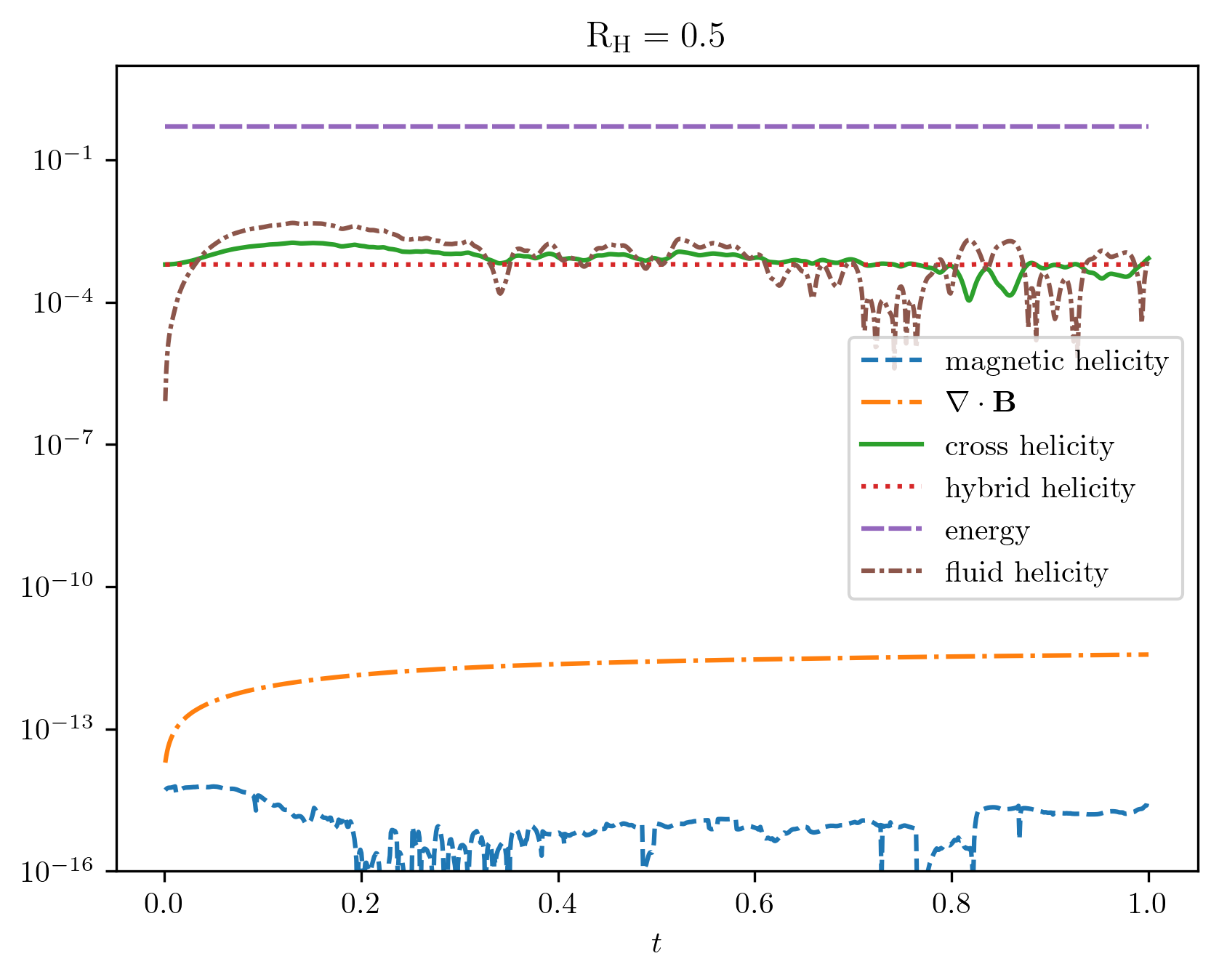} 
	\caption{Plot of conserved quantities for $\u \times \n = \mathbf{0}$ in the ideal limit.}
	\label{fig:uHcurl}
\end{figure}

\begin{figure}[htbp!]
	\centering
	\begin{tabular}{cc}
		\includegraphics[width=7.5cm]{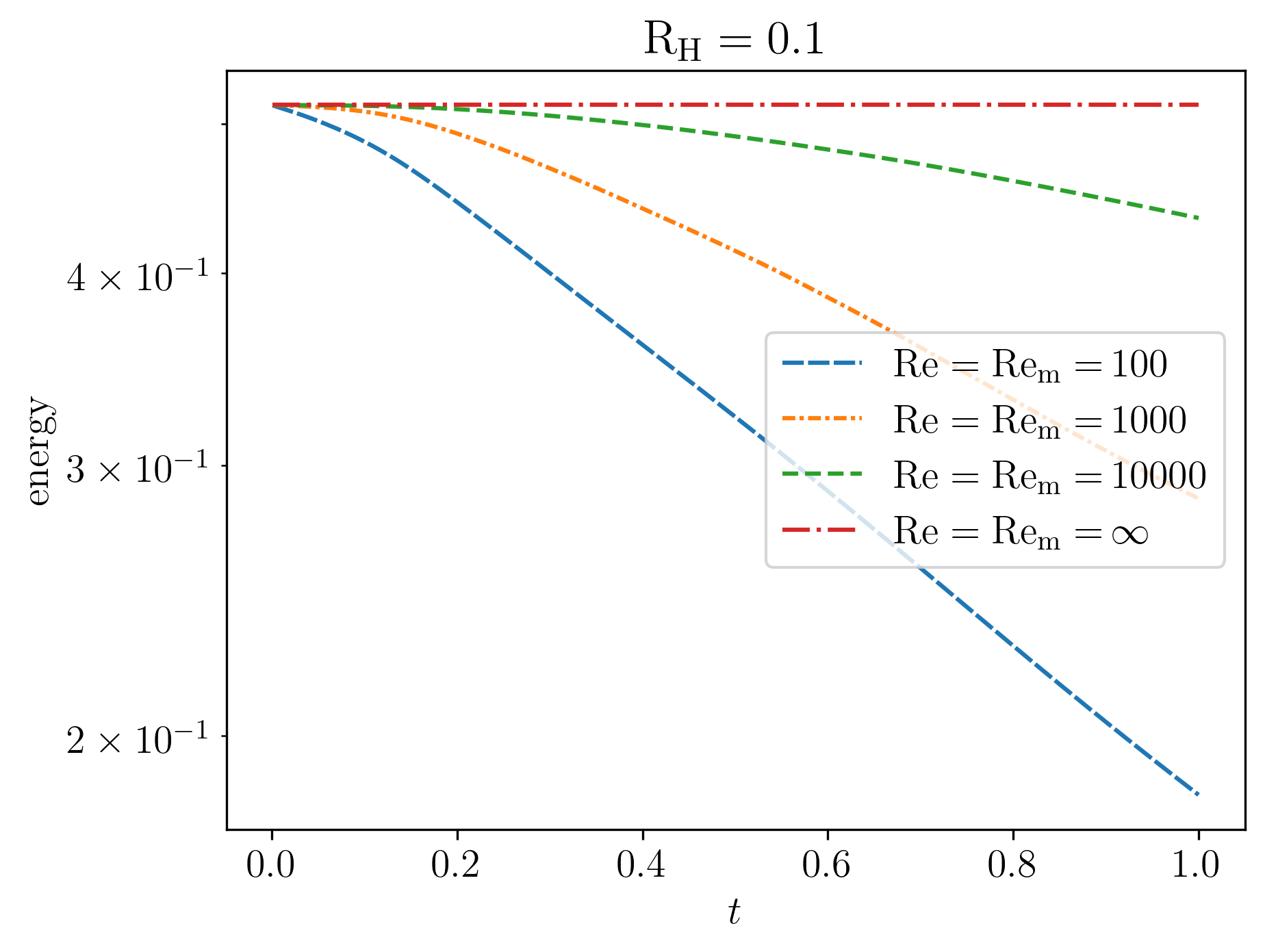} &
		\includegraphics[width=7.5cm]{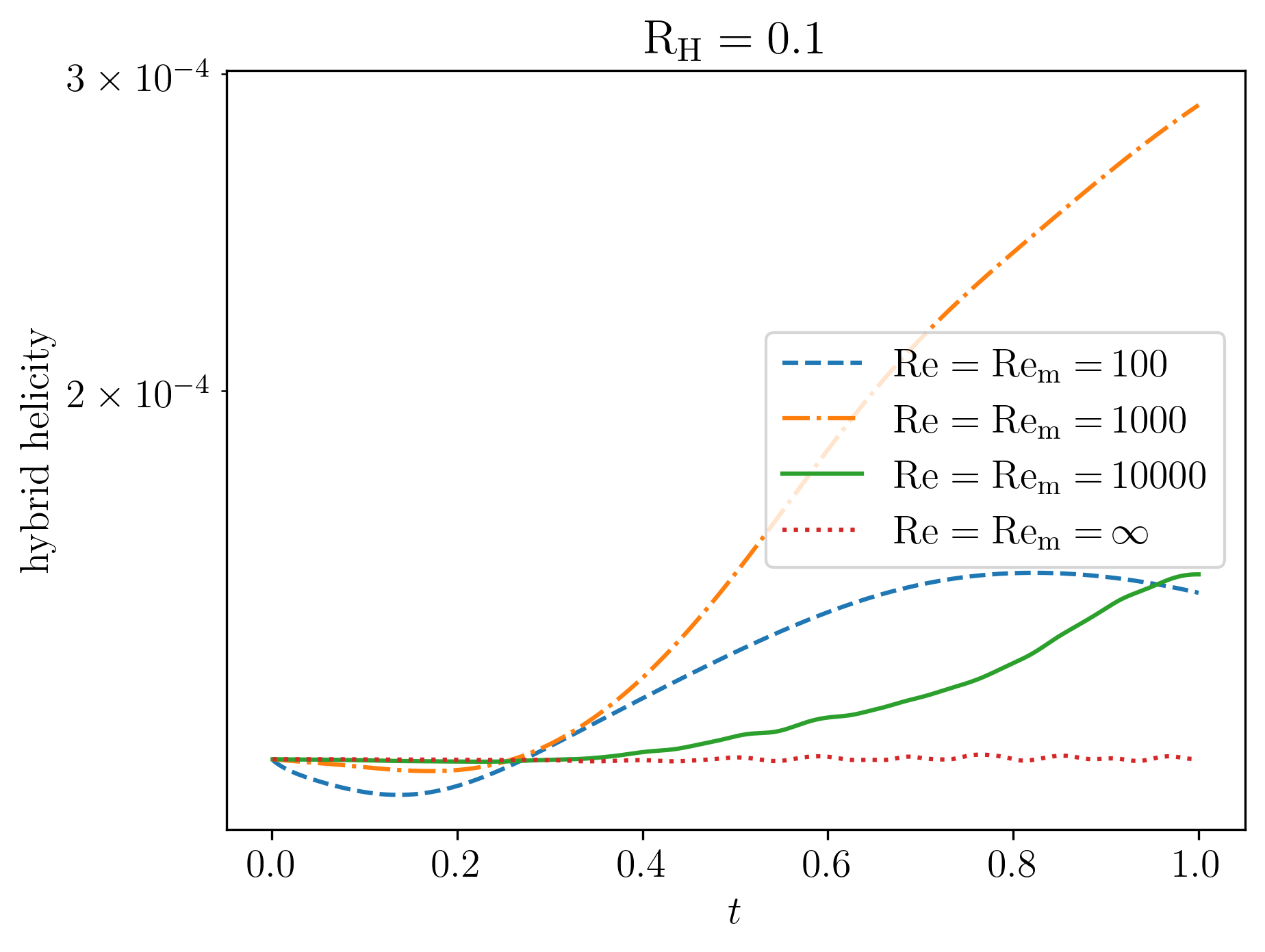}
	\end{tabular}
	\caption{Plots of the energy (left) and hybrid helicity (right) for different values of $\Re$ and $\Rem$ for $\u\times \n = \mathbf{0}$.}
	\label{fig:differentRe}
	
\end{figure}

\begin{figure}[htbp!]
	\centering
	\begin{tabular}{cc}
		\includegraphics[width=7.5cm]{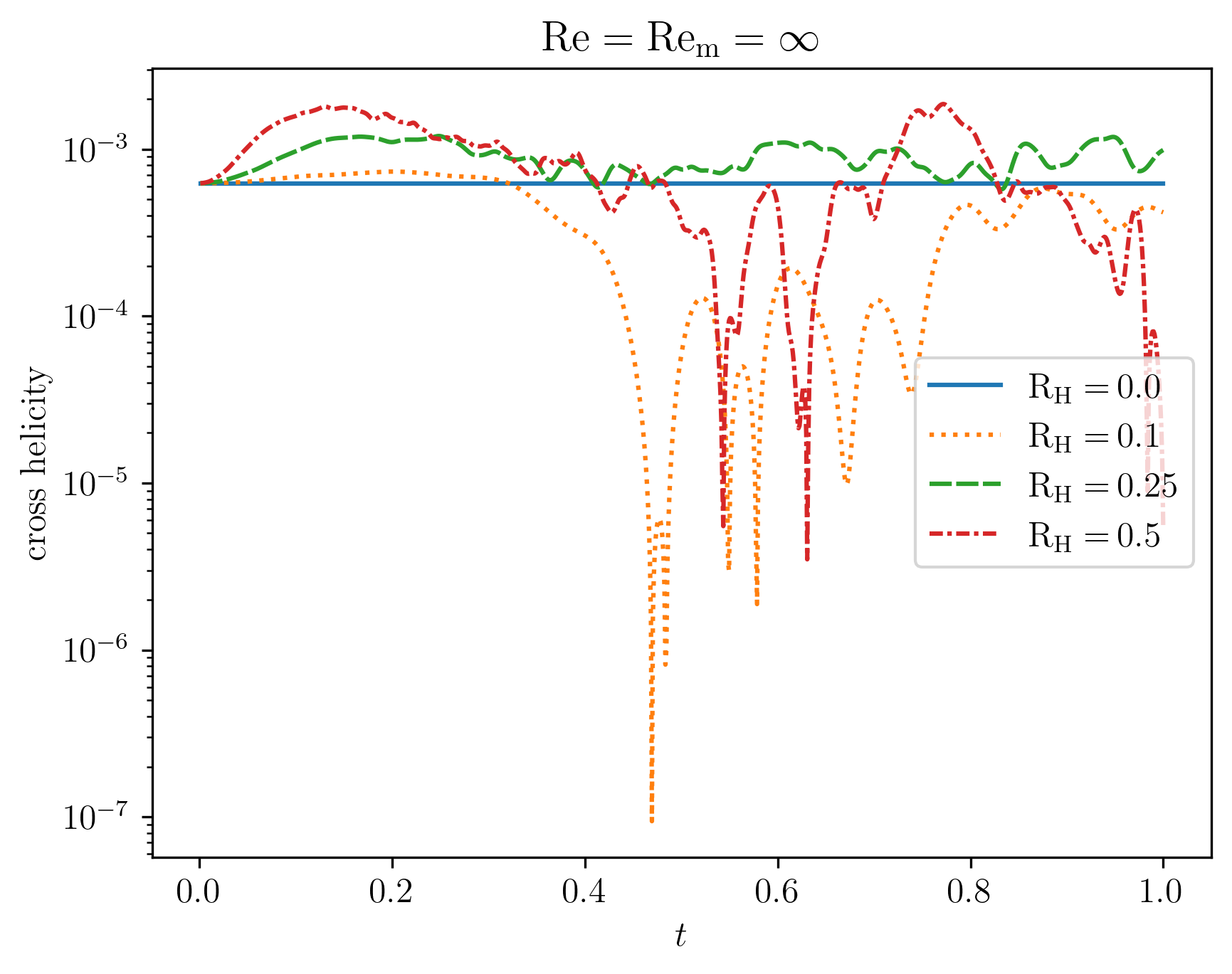} &
		\includegraphics[width=7.5cm]{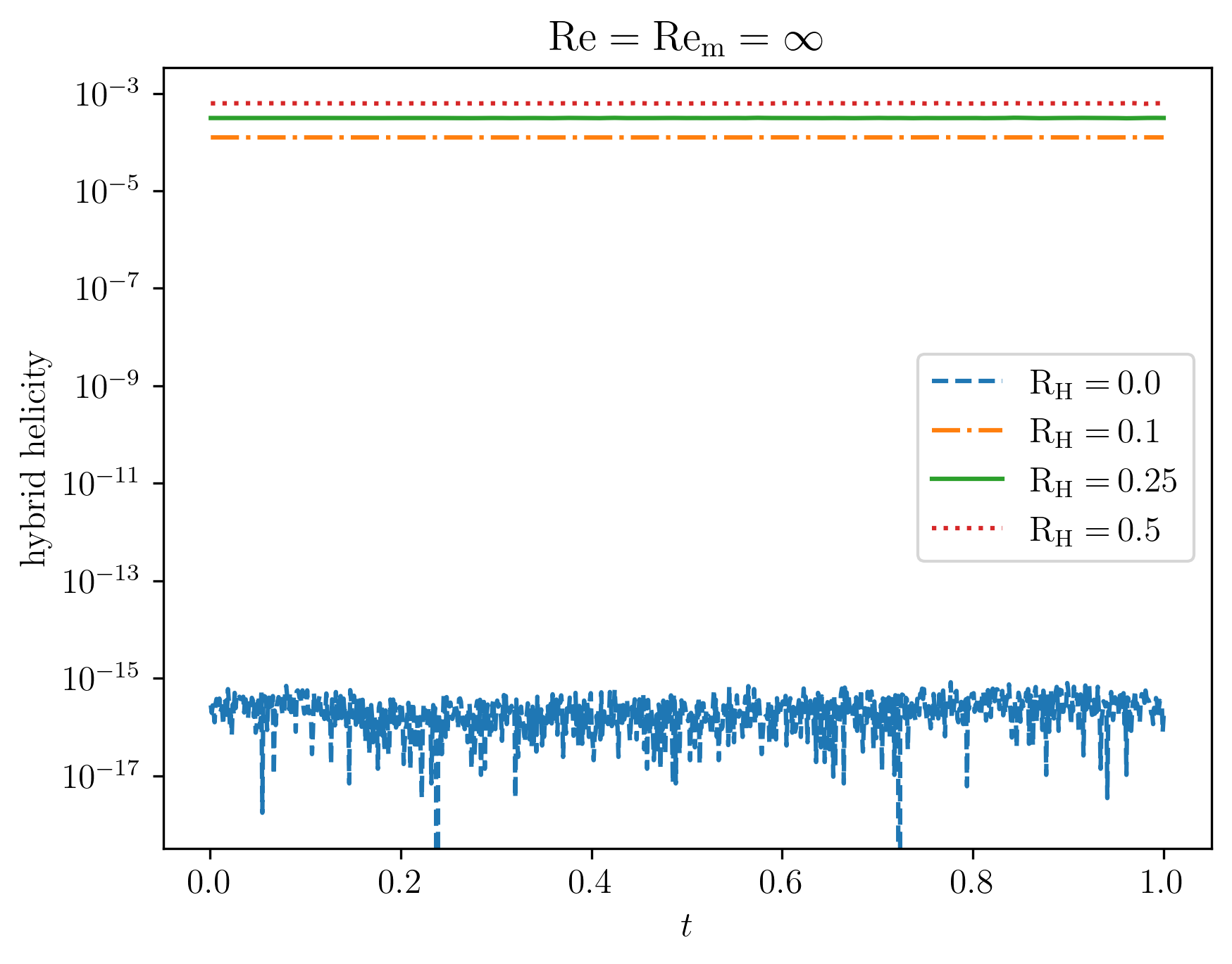}
	\end{tabular}
	\caption{Plots of the cross helicity (left) and hybrid helicity (right) for different values of $\RH$ for $\u\times \n = \mathbf{0}$ in the ideal limit of $\Re=\Rem=\infty$.}
	\label{fig:differentRH}
\end{figure}

\subsection{Test of conservative scheme for $\u \cdot \n = 0$ }
In this test, we  verify our results from Section \ref{sec:schemeucdotn} for the boundary conditions $\u \cdot \n = 0$. Here, we use the same initial conditions for $\B^0$ as before and 
\begin{gather}
	\begin{split}
		\u^0 = \nabla \times \v_{\text{pot}} \quad \text{with} \quad 
		\v_{\text{pot}}(x,y,z) = \frac{1}{\pi}\begin{pmatrix}
			\sin(\pi y)\sin(\pi z) \\ 
			\sin(\pi x)\sin(\pi z) \\ 
			\sin(\pi x)\sin(\pi y)
		\end{pmatrix},
	\end{split}
\end{gather}
which satisfy the boundary condition $\u^0\cdot \n = 0$ and $\nabla \cdot \u^0=0$. We discretize $\u$ with $\mathbb{RT}_1$-elements and $p$ with $\mathbb{DG}_0$-elements. We solve the system with a similar fixed point iteration to the one we described in the last subsection. The iteration coincides with that used in \cite[Section 6]{gawlik2020}.

 In contrast to the case $\u \times \n = \mathbf{0}$, we now enforce $\nabla \cdot \u_h = 0$ precisely over time. All conserved properties are plotted in Figure \ref{fig:uHdiv}. Remember that the hybrid helicity is not conserved for this scheme and therefore not displayed here. Moreover, corresponding plots to Figure \ref{fig:differentRe} and  \ref{fig:differentRH} show similar results and are therefore omitted here.
 
\begin{figure}[htbp!]
	\centering
	\includegraphics[width=12cm]{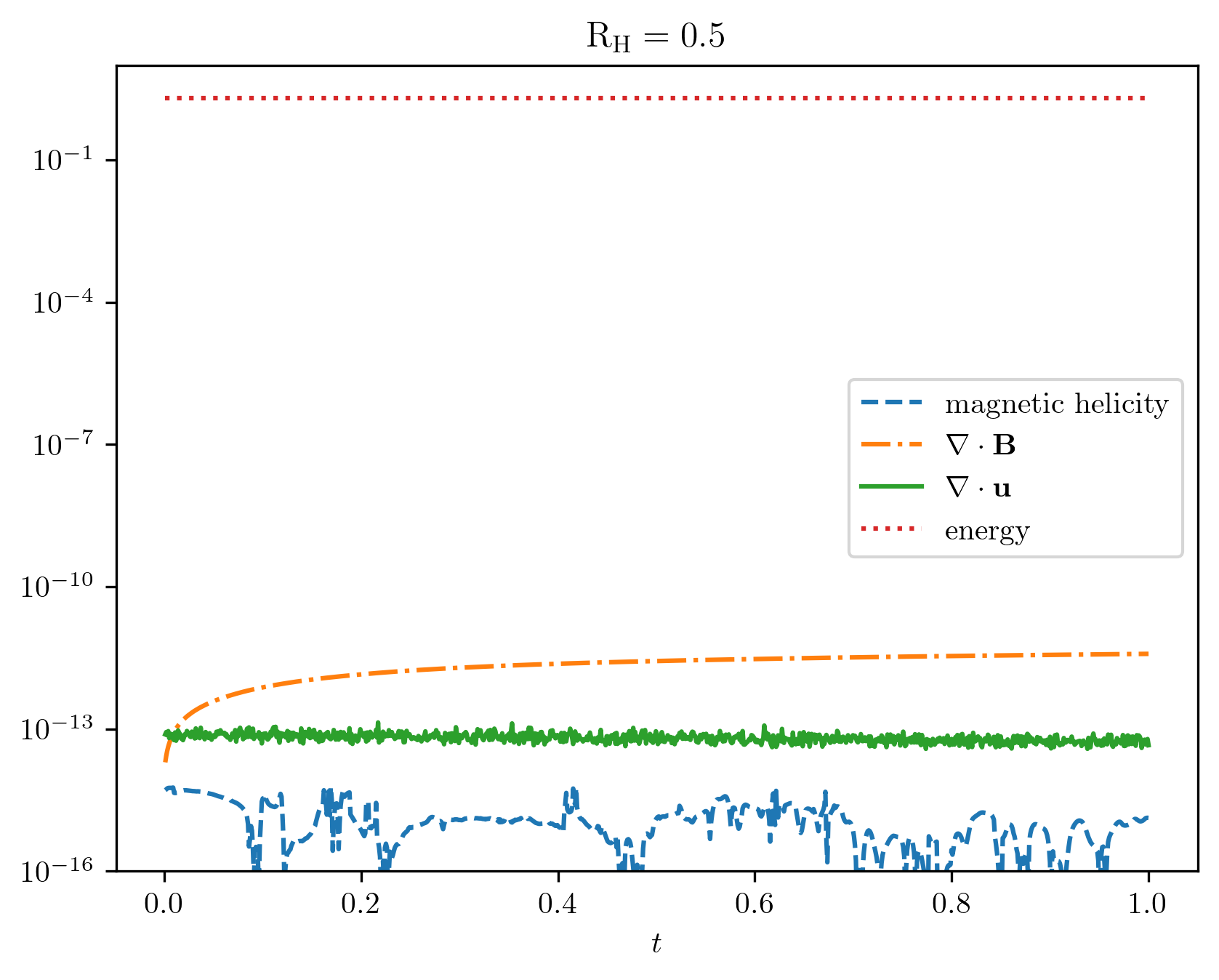} 
	\caption{Plot of conservative quantities for $\u \cdot \n = 0$ in the ideal limit.}
	\label{fig:uHdiv}
\end{figure}

\subsection{Island coalescence problem}
Finally, we consider a 2.5-dimensional island coalescence problem to model a magnetic reconnection process in large aspect ratio tokamaks. For a strong magnetic field in the toroidal direction, the flow can be described in a two-dimensional setting by considering a cross-section of the tokamak. We consider a similar problem as in \cite[Section 4.2]{adler2020monolithic}. The domain $\Omega=(-1,1)^2$ results from the unfolding of an annulus in the cross-sectional direction where the left and right edges are mapped periodically. The equilibrium solution for $k=0.2$ is given by
\begin{align*}
	\u_{eq}=\mathbf{0},\qquad p_{eq}(x,y) = \frac{1-k^2}{2}\left (1+\frac{1}{(\cosh(2\pi y) + k \cos(2\pi x))^2}\right), \\
	\B_{eq}(x,y) = \frac{1}{\cosh (2\pi y) + k \cos(2\pi x)} 
	\begin{pmatrix}
		\sinh (2\pi y) \\ k \sin(2\pi x)
	\end{pmatrix}, \\
u_{3, eq} = B_{3, eq}= 0, \qquad \tilde{\j}_{eq} = \mathbf{0}, \qquad j_{3,eq} = \scurl \tilde{\B}_{eq},
\end{align*}
which results in right-hand sides $\f=\mathbf{0}$ and $\mathbf{g}$ given by
\begin{equation}
	\mathbf{g} = \frac{-8 \pi^2 (k^2-1)}{\Rem(\cosh(2\pi y) + k\cos(2\pi x))^3}
	\begin{pmatrix}
		\sinh(2\pi y) \\ k \sin(2\pi x)
	\end{pmatrix}.
\end{equation}
The components $\tilde{\E}_{eq}$ and $E_{3,eq}$ of the electric field are computed by the equations \eqref{eq:islandcoal-j} and \eqref{eq:islandcoal-j3}. 
The initial condition for $\B_{eq}$ is given by perturbing it for $\varepsilon=0.01$ with 
\begin{equation}
	\Delta \B = \frac{\varepsilon}{\pi} \begin{pmatrix}
		-\cos(\pi x) \sin(\pi y/2) \\
		2\cos(\pi y/2) \sin(\pi x)
	\end{pmatrix}.
\end{equation}
The authors believe that the reported $\Delta \B$ in \cite{adler2020monolithic} includes a typo, as it is not divergence-free, and amended the second component appropriately.
The reconnection rate can be computed as the difference between $\curl \B$ evaluated the origin $(0,0)$ at the current time and the initial time, divided by $\sqrt{\Rem}$. 
In order to make sense of the point evaluation of $j_0$ at $(0,0)$, we project $j_0$ to the space $\mathbb{CG}1$ as in \cite{adler2020monolithic}.  For the additional variables, we set the equilibrium solution
\begin{align}
	u_{3, eq} = B_{3, eq}= 0, \qquad \tilde{\j}_{eq} = \mathbf{0}, \qquad j_{3,eq} = \scurl \tilde{\B}_{eq}. 
\end{align}
Since we use a direct solver for the solution of the Schur complement, we only considered a base mesh $20\times20$ cells and three levels of refinement resulting in an $160\times160$ mesh. We iterated until the final time $T=12.0$ with a step size of $\Delta t = 0.025$.

Figure \ref{fig:reconHall} shows the reconnection rate for different choices of $\RH$ at $\Rem=\Re=100,500,1000,1500$. All graphs have in common that the reconnection process happens faster for higher Hall parameters. This is consistent with the results of other numerical experiments \cite[Section 4.3]{Morales2005}\cite{Huba2003}}. For $\Rem=\Re=100,500$ one can observe that the heights of the peaks increases with growing Hall parameters. At $\Rem = \Re = 1000$ this trend is broken and for $\Rem=\Re=1500$ the height of the peaks starts to decrease for higher Hall parameters. Furthermore, additional peaks occur for high Hall parameters and Reynolds numbers. 

\begin{figure}[htbp!]
\centering
\begin{tabular}{cc}
\includegraphics[width=7.5cm]{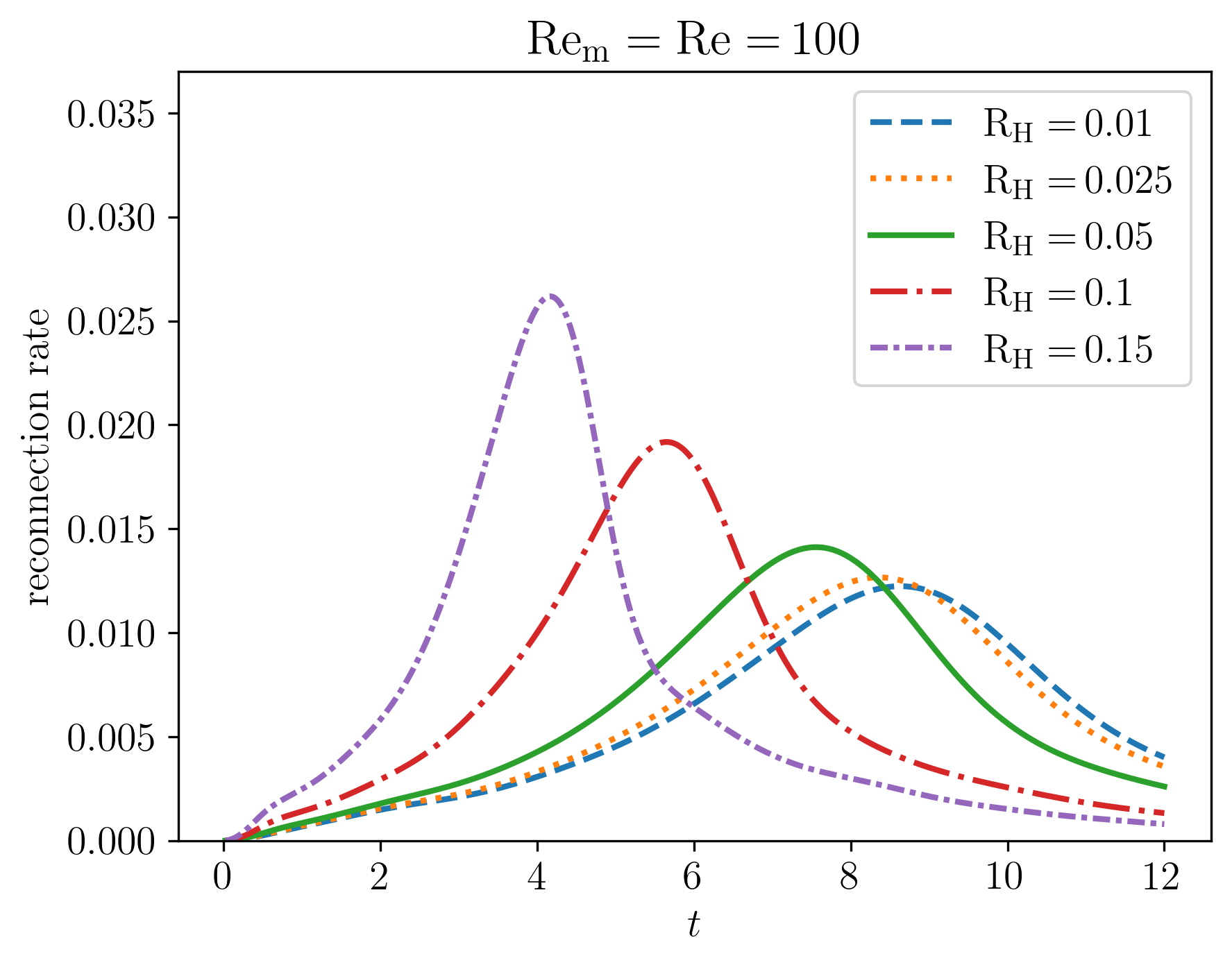} &
\includegraphics[width=7.5cm]{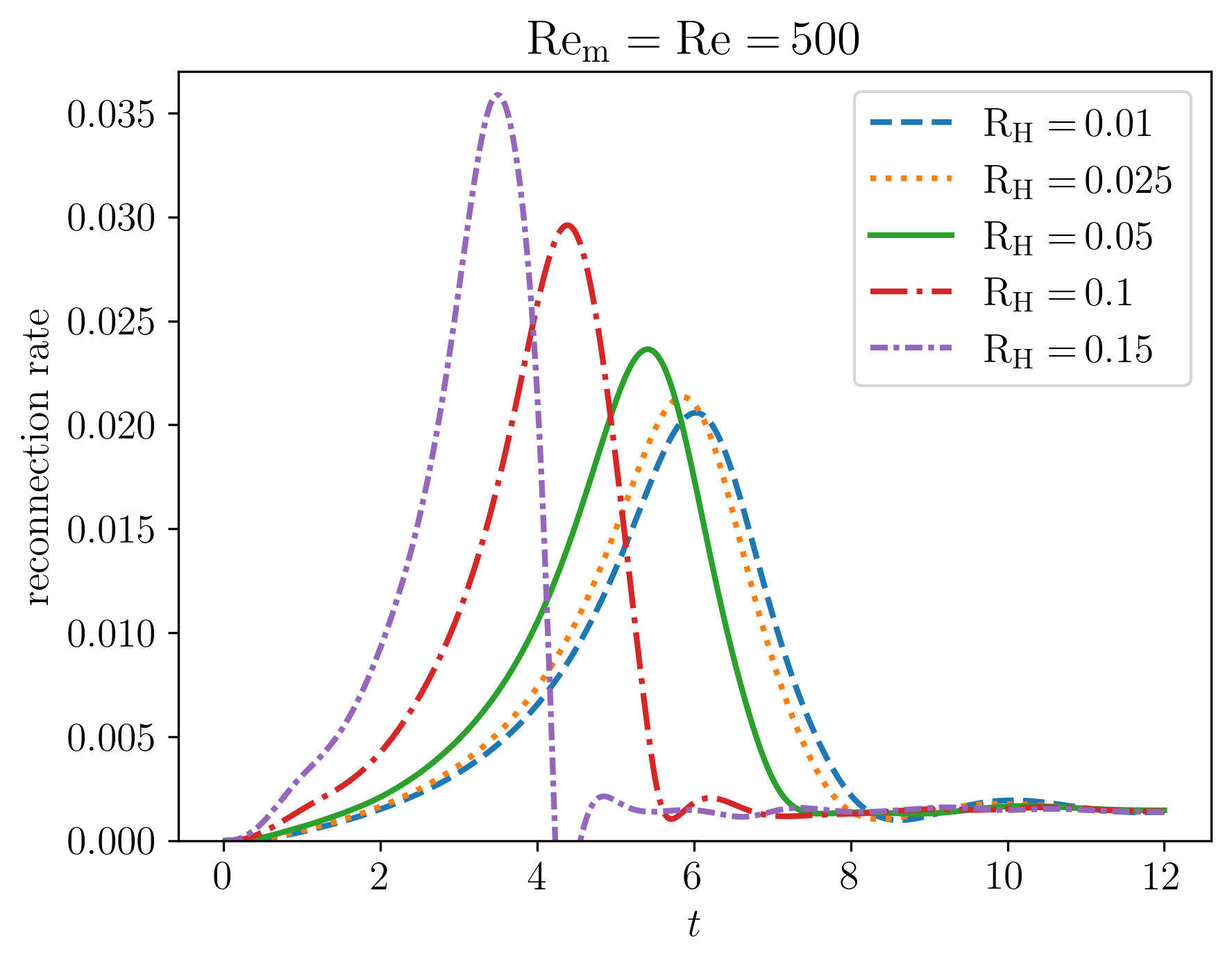} \\
\includegraphics[width=7.5cm]{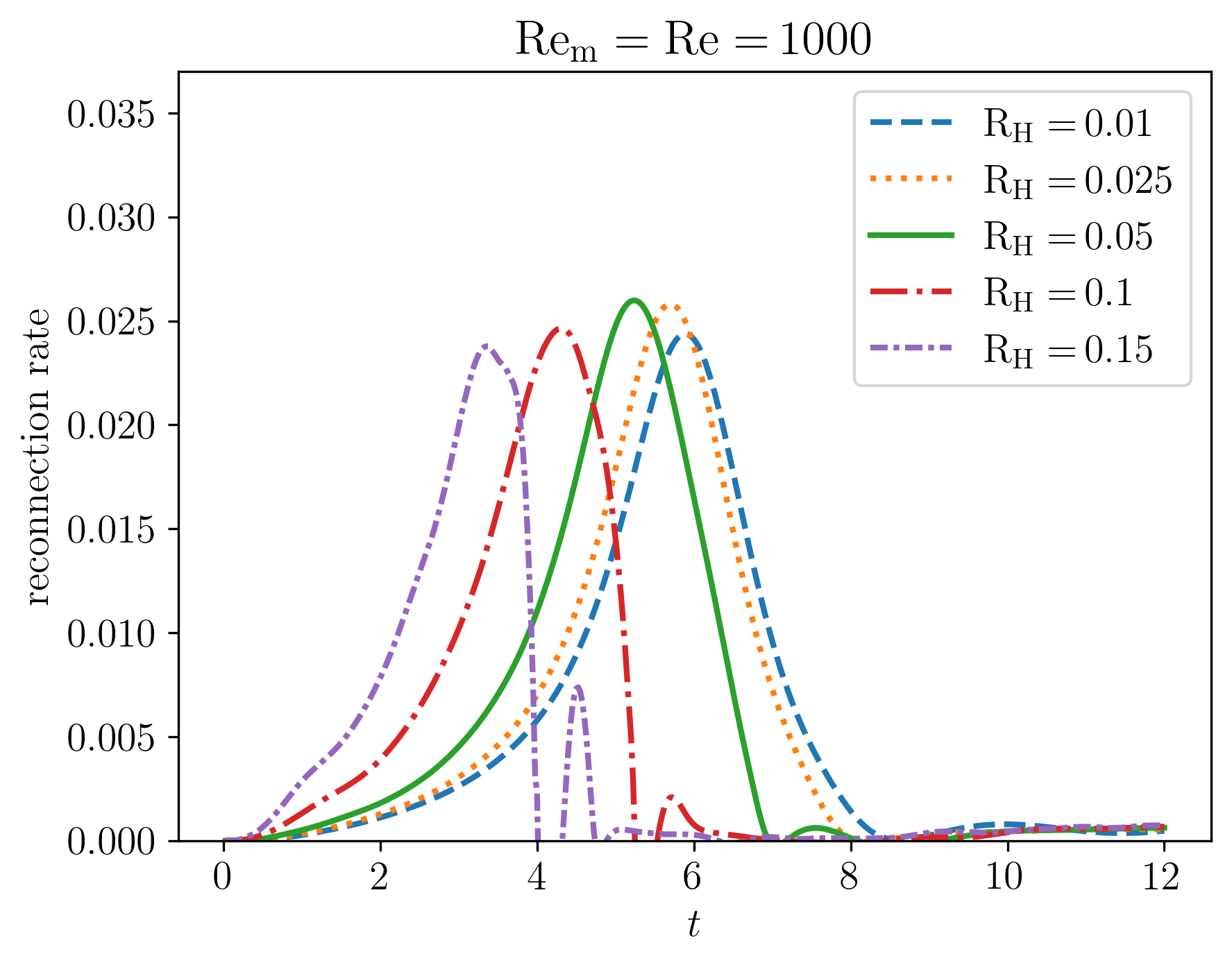} &
\includegraphics[width=7.5cm]{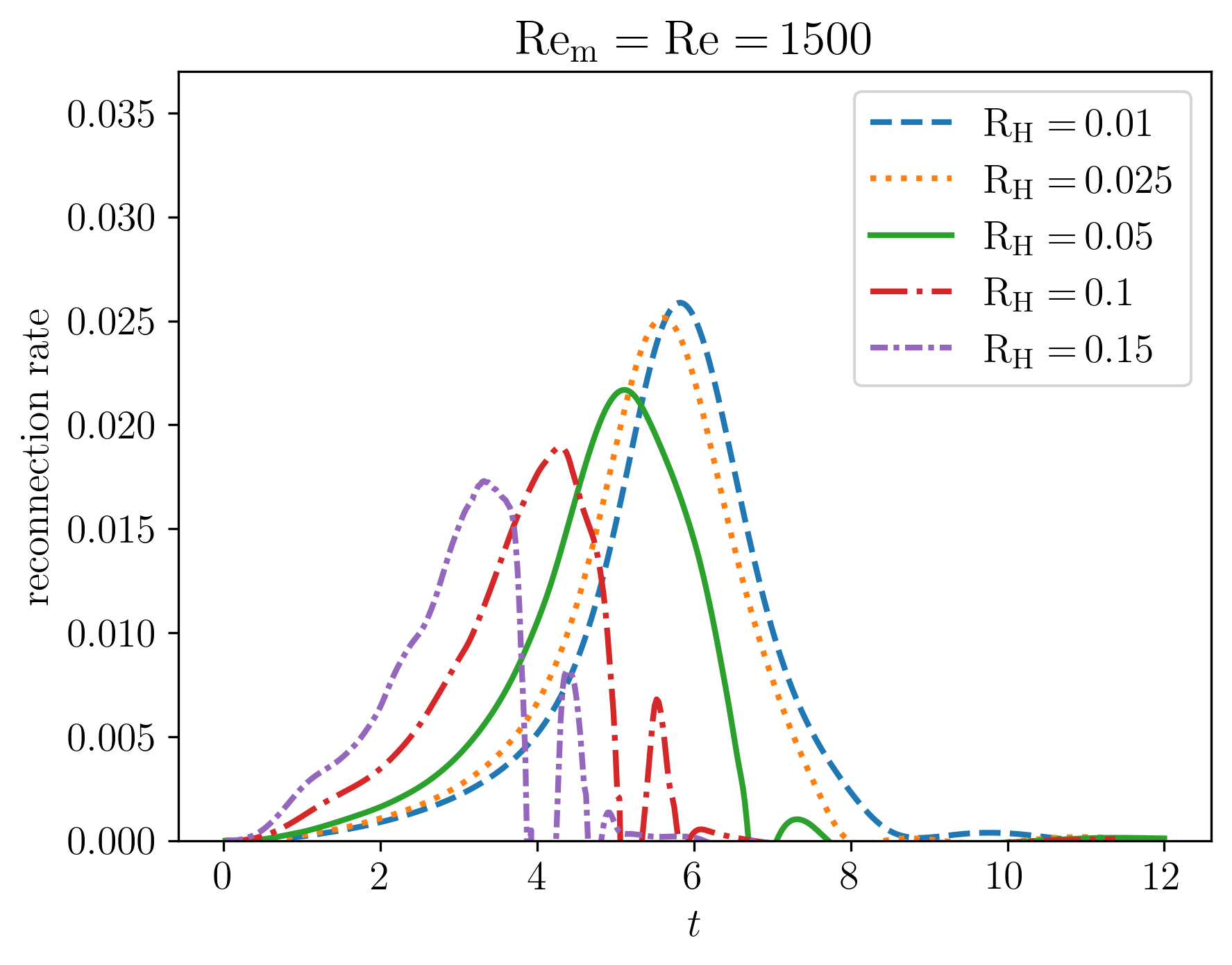} \\
\end{tabular}
\caption{Reconnection rates for an island coalescence problem for different choices of $\RH$.}
\label{fig:reconHall}
\end{figure}

\section{Conclusion and Outlook} We have presented a structure-preserving finite element discretization for the incompressible Hall MHD equations that enforces $\nabla \cdot \B = 0$ precisely and proved the well-posedness and convergence of a Picard-type linearization. Furthermore, we presented formulations that preserve the energy, magnetic and hybrid helicity precisely on the discrete level in the ideal limit for two types of boundary conditions. Finally, we investigated a block preconditioning strategy that works well as long as $\RH$ and $S$ or $\Rem$ are not chosen too high at the same time.

 In future work, we want to improve the robustness of our solver with respect to the Hall parameter, especially in the 2.5-dimensional case where we currently use a direct solver to solve the electromagnetic block. This would also enable us to consider the island coalescence problem on much finer grids.  Furthermore, we are curious to investigate further if there exists a scheme that also preserves the hybrid helicity at the same time as the other quantities for the case $\u\cdot \n = 0$.

 \section*{Code availability} The code that was used to generate the numerical results and all major Firedrake components have been archived on \cite{zenodo/Firedrake-20220223.0}.

\bibliographystyle{elsarticle-num}
\bibliography{reference}{}  

\begin{thebibliography}{10}
\expandafter\ifx\csname url\endcsname\relax
  \def\url#1{\texttt{#1}}\fi
\expandafter\ifx\csname urlprefix\endcsname\relax\def\urlprefix{URL }\fi
\expandafter\ifx\csname href\endcsname\relax
  \def\href#1#2{#2} \def\path#1{#1}\fi

\bibitem{Gerbeau2006}
J.-F. Gerbeau, C.~L. Bris, T.~Leli{\`{e}}vre, Mathematical Methods for the
  Magnetohydrodynamics of Liquid Metals, Oxford University Press, 2006.

\bibitem{gunzburger1991existence}
M.~D. Gunzburger, A.~J. Meir, J.~S. Peterson, {On the existence, uniqueness,
  and finite element approximation of solutions of the equations of stationary,
  incompressible magnetohydrodynamics}, Mathematics of Computation 56~(194)
  (1991) 523--563.

\bibitem{Galtier2015}
S.~Galtier, Introduction to Modern Magnetohydrodynamics, Cambridge University
  Press, 2015.

\bibitem{Huba2003}
J.~D. Huba, Hall Magnetohydrodynamics - A Tutorial, Springer Berlin Heidelberg,
  2003, pp. 166--192.

\bibitem{Forbes1991}
T.~G. Forbes, Magnetic reconnection in solar flares, Geophysical {\&}
  Astrophysical Fluid Dynamics 62~(1-4) (1991) 15--36.

\bibitem{Morales2005}
L.~F. Morales, S.~Dasso, D.~O. G{\'{o}}mez, Hall effect in incompressible
  magnetic reconnection, Journal of Geophysical Research: Space Physics
  110~(A4) (2005).

\bibitem{Ripin1993}
B.~H. Ripin, J.~D. Huba, E.~A. McLean, C.~K. Manka, T.~Peyser, H.~R. Burris,
  J.~Grun, Sub-{A}lfv{\'{e}}nic plasma expansion, Physics of Fluids B: Plasma
  Physics 5~(10) (1993) 3491--3506.

\bibitem{Chae2014}
D.~Chae, P.~Degond, J.-G. Liu, Well-posedness for {H}all-magnetohydrodynamics,
  Annales de l{\textquotesingle}Institut Henri Poincare (C) Non Linear Analysis
  31~(3) (2014) 555--565.

\bibitem{Danchin2020}
R.~Danchin, J.~Tan, On the well-posedness of the {H}all-magnetohydrodynamics
  system in critical spaces, Communications in Partial Differential Equations
  46~(1) (2020) 31--65.

\bibitem{Gmez2008}
D.~O. G{\'{o}}mez, S.~M. Mahajan, P.~Dmitruk, Hall magnetohydrodynamics in a
  strong magnetic field, Physics of Plasmas 15~(10) (2008) 102303.

\bibitem{Chacn2003}
L.~Chac{\'{o}}n, D.~Knoll, A 2d high-$\beta$ {H}all {MHD} implicit nonlinear
  solver, Journal of Computational Physics 188~(2) (2003) 573--592.

\bibitem{Tth2008}
G.~T{\'{o}}th, Y.~Ma, T.~I. Gombosi, Hall magnetohydrodynamics on
  block-adaptive grids, Journal of Computational Physics 227~(14) (2008)
  6967--6984.

\bibitem{Brackbill1980}
J.~Brackbill, D.~Barnes, The effect of nonzero {$\nabla$ $\cdotp \mathrm{B}$}
  on the numerical solution of the magnetohydrodynamic equations, Journal of
  Computational Physics 35~(3) (1980) 426--430.

\bibitem{Hu20162}
K.~Hu, Y.~Ma, J.~Xu, Stable finite element methods preserving {$\nabla \cdot B
  = 0$} exactly for {MHD} models, Numerische Mathematik 135~(2) (2016)
  371--396.

\bibitem{hu2019structure}
K.~Hu, J.~Xu, Structure-preserving finite element methods for stationary {MHD}
  models, Mathematics of Computation 88~(316) (2019) 553--581.

\bibitem{hu2020convergence}
K.~Hu, W.~Qiu, K.~Shi, {Convergence of a BE based finite element method for MHD
  models on Lipschitz domains}, Journal of Computational and Applied
  Mathematics 368 (2020) 112477.

\bibitem{adler2018vector}
J.~H. Adler, Y.~He, X.~Hu, S.~P. MacLachlan, Vector-potential finite-element
  formulations for two-dimensional resistive magnetohydrodynamics, Computers \&
  Mathematics with Applications (2018).

\bibitem{hiptmair2018fully}
R.~Hiptmair, L.~Li, S.~Mao, W.~Zheng, A fully divergence-free finite element
  method for magnetohydrodynamic equations, Mathematical Models and Methods in
  Applied Sciences 28~(04) (2018) 659--695.

\bibitem{pagliantini2016computational}
C.~Pagliantini, {Computational Magnetohydrodynamics with Discrete Differential
  Forms}, Ph.D. thesis (2016).

\bibitem{Mininni_2003}
P.~D. Mininni, D.~O. Gomez, S.~M. Mahajan, Dynamo action in
  magnetohydrodynamics and {H}all-magnetohydrodynamics, The Astrophysical
  Journal 587~(1) (2003) 472--481.

\bibitem{gawlik2020}
E.~S. Gawlik, F.~Gay-Balmaz, A finite element method for {MHD} that preserves
  energy, cross-helicity, magnetic helicity, incompressibility, and
  div {B} = 0, Journal of Computational Physics 450 (2022) 110847.

\bibitem{hu2021helicity}
K.~Hu, Y.-J. Lee, J.~Xu, {Helicity-conservative finite element discretization
  for incompressible MHD systems}, Journal of Computational Physics 436 (2021)
  110284.

\bibitem{moffatt1992helicity}
H.~Moffatt, A.~Tsinober, Helicity in laminar and turbulent flow, Annual review
  of fluid mechanics 24~(1) (1992) 281--312.

\bibitem{taylor1974relaxation}
B.~J. Taylor, Relaxation of toroidal plasma and generation of reverse magnetic
  fields, Physical Review Letters 33~(19) (1974) 1139.

\bibitem{pariat2005photospheric}
E.~Pariat, P.~D{\'e}moulin, M.~Berger, Photospheric flux density of magnetic
  helicity, Astronomy \& Astrophysics 439~(3) (2005) 1191--1203.

\bibitem{perez2009role}
J.~C. Perez, S.~Boldyrev, Role of cross-helicity in magnetohydrodynamic
  turbulence, Physical review letters 102~(2) (2009) 025003.

\bibitem{arnold1999topological}
V.~I. Arnold, B.~A. Khesin, {Topological methods in hydrodynamics}, Vol. 125,
  Springer Science \& Business Media, 1999.

\bibitem{laakmann2021}
F.~Laakmann, P.~E. Farrell, L.~Mitchell, An augmented {L}agrangian
  preconditioner for the magnetohydrodynamics equations at high {R}eynolds and
  coupling numbers, arXiv preprint arXiv:2104.14855 (2021).

\bibitem{Phillips2016}
E.~G. Phillips, J.~N. Shadid, E.~C. Cyr, H.~C. Elman, R.~P. Pawlowski, Block
  preconditioners for stable mixed nodal and edge finite element
  representations of incompressible resistive {MHD}, SIAM Journal on Scientific
  Computing 38~(6) (2016) B1009--B1031.

\bibitem{Wathen2020}
M.~Wathen, C.~Greif, A scalable approximate inverse block preconditioner for an
  incompressible magnetohydrodynamics model problem, {SIAM} Journal on
  Scientific Computing 42~(1) (2020) B57--B79.

\bibitem{adler2020monolithic}
J.~H. Adler, T.~R. Benson, E.~C. Cyr, P.~E. Farrell, S.~P. MacLachlan, R.~S.
  Tuminaro, Monolithic multigrid for magnetohydrodynamics, SIAM Journal on
  Scientific Computing (2021).

\bibitem{shadid2016scalable}
J.~N. Shadid, R.~P. Pawlowski, E.~C. Cyr, R.~S. Tuminaro, L.~Chac\'on,
  P.~Weber, {Scalable implicit incompressible resistive MHD with stabilized FE
  and fully-coupled Newton-Krylov-AMG}, Computer Methods in Applied Mechanics
  and Engineering 304 (2016) 1--25.

\bibitem{Donato2012}
S.~Donato, S.~Servidio, P.~Dmitruk, V.~Carbone, M.~A. Shay, P.~A. Cassak, W.~H.
  Matthaeus, Reconnection events in two-dimensional {H}all magnetohydrodynamic
  turbulence, Physics of Plasmas 19~(9) (2012) 092307.

\bibitem{Shi2019}
C.~Shi, A.~Tenerani, M.~Velli, S.~Lu, Fast recursive reconnection and the
  {H}all effect: Hall-{MHD} simulations, The Astrophysical Journal 883~(2)
  (2019) 172.

\bibitem{Arnold.D;Falk.R;Winther.R.2006a}
D.~N. Arnold, R.~S. Falk, R.~Winther, Finite element exterior calculus,
  homological techniques, and applications, Acta numerica 15 (2006) 1--155.

\bibitem{Arnold.D;Falk.R;Winther.R.2010a}
D.~N. Arnold, R.~S. Falk, R.~Winther, Finite element exterior calculus: from
  hodge theory to numerical stability, Bulletin of the American mathematical
  society 47~(2) (2010) 281--354.

\bibitem{Hiptmair.R.2002a}
R.~Hiptmair, Finite elements in computational electromagnetism, Acta Numerica
  11 (2002) 237--339.

\bibitem{Bossavit.A.1998a}
A.~Bossavit, Computational electromagnetism: variational formulations,
  complementarity, edge elements, Academic Press, 1998.

\bibitem{Girault.V;Raviart.P.1986a}
V.~Girault, P.-A. Raviart, Finite element methods for {Navier-Stokes}
  equations: theory and algorithms, Vol.~5, Springer Science \& Business Media,
  2012.

\bibitem{he2019generalized}
J.~He, K.~Hu, J.~Xu, {Generalized Gaffney inequality and discrete compactness
  for discrete differential forms}, Numerische Mathematik 143~(4) (2019)
  781--795.

\bibitem{Brezzi.F.1974a}
F.~Brezzi, {On the existence, uniqueness and approximation of saddle-point
  problems arising from Lagrangian multipliers}, ESAIM: Mathematical Modelling
  and Numerical Analysis-Mod\'{e}lisation Math\'{e}matique et Analyse
  Num\'{e}rique 8 (1974) 129--151.

\bibitem{ma2016robust}
Y.~Ma, K.~Hu, X.~Hu, J.~Xu, {Robust preconditioners for incompressible MHD
  models}, Journal of Computational Physics 316 (2016) 721--746.

\bibitem{moffatt2021some}
H.~Moffatt, Some topological aspects of fluid dynamics, Journal of Fluid
  Mechanics 914 (2021).

\bibitem{schoberl1999b}
J.~Sch\"oberl, Robust multigrid methods for parameter dependent problems, Ph.D.
  thesis, Johannes Kepler Universit\"at Linz (1999).

\bibitem{rathgeber2016}
F.~Rathgeber, D.~A. Ham, L.~Mitchell, M.~Lange, F.~Luporini, A.~T.~T. Mcrae,
  G.-T. Bercea, G.~R. Markall, P.~H.~J. Kelly, Firedrake: automating the finite
  element method by composing abstractions, {ACM} Transactions on Mathematical
  Software 43~(3) (2016) 1--27.

\bibitem{balay2019}
S.~Balay, S.~Abhyankar, M.~F. Adams, J.~Brown, P.~Brune, K.~Buschelman,
  L.~Dalcin, V.~Eijkhout, W.~D. Gropp, D.~Karpeyev, D.~Kaushik, M.~G. Knepley,
  D.~A. May, L.~C. McInnes, R.~T. Mills, T.~Munson, K.~Rupp, P.~Sanan, B.~F.
  Smith, S.~Zampini, H.~Zhang, H.~Zhang, {PETS}c users manual, Tech. Rep.
  ANL-95/11 - Revision 3.15, Argonne National Laboratory (2021).

\bibitem{farrell2019pcpatch}
P.~E. Farrell, M.~G. Knepley, L.~Mitchell, F.~Wechsung, {PCPATCH: software for
  the topological construction of multigrid relaxation methods}, ACM
  Transactions on Mathematical Software (2021).

\bibitem{Quarteroni2017}
A.~Quarteroni, Numerical Models for Differential Problems, Springer
  International Publishing, 2017.

\bibitem{boffi2013mixed}
D.~Boffi, F.~Brezzi, M.~Fortin, Mixed finite elements for electromagnetic
  problems, in: Mixed Finite Element Methods and Applications, Springer, 2013,
  pp. 625--662.

\bibitem{Ipsen1998}
I.~C.~F. Ipsen, C.~D. Meyer, The idea behind {K}rylov methods, The American
  Mathematical Monthly 105~(10) (1998) 889--899.

\bibitem{zenodo/Firedrake-20220223.0}
{Software used in `Structure-preserving and helicity-conserving finite element
  approximations and preconditioning for the Hall MHD equations'} (Feb 2022).
\newblock \href {https://doi.org/10.5281/zenodo.6243332}
  {\path{doi:10.5281/zenodo.6243332}}.

\end{thebibliography}

\end{document}